\documentclass[a4paper,12pt]{article}

\usepackage{amsmath, amsthm, amsfonts}
\usepackage[latin1]{inputenc}

\newcommand{\abs}[1]{\left\vert#1\right\vert}
\newcommand{\ap}[1]{\left\langle#1\right\rangle}
\newcommand{\norm}[1]{\left\Vert#1\right\Vert}
\newcommand{\tnorm}[1]{\left\vert\!\left\vert\!\left\vert#1\right\vert\!\right\vert\!\right\vert}

\def\RR{\mathbb{R}}

\newtheorem{thm}{Theorem}[section]
\newtheorem{cor}[thm]{Corollary}
\newtheorem{lem}[thm]{Lemma}
\newtheorem{prp}[thm]{Proposition}

\theoremstyle{definition}
\newtheorem{dfn}[thm]{Definition}
\theoremstyle{remark}
\newtheorem{rem}[thm]{Remark}

\title{Rate of convergence to self-similarity for Smoluchowski's
  coagulation equation with constant coefficients}

\author{
  José A. Cañizo
  \thanks{Supported by the ANR research group
    SPINADA and a postdoc grant from the Spanish
    \emph{Ministerio de Educación y Ciencia}.
    Email:  \texttt{canizo@ceremade.dauphine.fr}
  }
  \\
  Stéphane Mischler
  \thanks{Email: \texttt{mischler@ceremade.dauphine.fr}}
  \\
  Clément Mouhot
  \thanks{Email: \texttt{mouhot@ceremade.dauphine.fr}}
  \\
  CEREMADE, Univ. Paris-Dauphine
  \\
  Place du Maréchal de Lattre de Tassigny
  \\
  75775 Paris CEDEX 16, France}

\begin{document}
\maketitle

\abstract{We show that solutions to Smoluchowski's equation with a
  constant coagulation kernel and an initial datum with some
  regularity and exponentially decaying tail converge exponentially
  fast to a self-similar profile. This convergence holds in a weighted
  Sobolev norm which implies the $L^2$ convergence of derivatives up
  to a certain order $k$ depending on the regularity of the initial
  condition. We prove these results through the study of the
  linearized coagulation equation in self-similar variables, for which
  we show a spectral gap in a scale of weighted Sobolev spaces. We
  also take advantage of the fact that the Laplace or Fourier
  transforms of this equation can be explicitly solved in this case.}

\vspace{0.2cm}


\vspace{0.2cm}


\newpage
\tableofcontents

\section{Introduction}

\subsection{Smoluchowski's equation with a constant kernel}

Smoluchowski's equation is a well-known model for the time evolution
of irreversible aggregation processes
\cite{D72, A99, citeulike:1134194, MR2068589}. The object of our study is
the continuous version of this equation, with the additional
assumption that the rate at which any two nuclei coalesce (the
\emph{coagulation rate} or \emph{kernel}) is independent of their
sizes. If the number of nuclei of size $y > 0$ at a given time $t$ is
given by $f(t,y)\,dy$, then Smoluchowski's equation reads:
\begin{equation}
  \label{eq:coag-eq}
  \partial_t f(t,y) = C(f(t,\cdot), f(t,\cdot))(y),
\end{equation}
where
\begin{equation}
  \label{eq:coag_operator}
    C(f,f)(y) := 
    \frac{1}{2} \int_0^y f(x) f(y-x) \,dx
    - f(y) \int_0^\infty f(x) \,dx.
\end{equation}
When other coagulation rates are considered, a coagulation coefficient
$a(x,y)$ appears in the equation; here we have set $a(x,y) = 1$ for
all $x,y > 0$. The integral $\int_0^\infty f(t,y) \,y \,dy$, called
\emph{the mass of $f$} at time $t$, represents the total mass of all
nuclei at time $t$, and is a conserved quantity for this equation.

One of the more remarkable properties of Smoluchowski's equation is
the self-similar behavior that its solutions exhibit at large times,
which is expected by formal arguments but has been rigorously proved
only in very particular cases. This behavior is our main concern in
this paper, and among our main results we highlight the following one:
\begin{thm}
  \label{thm:uniform convergence - intro}
  Take a locally absolutely continuous function $f_0 :
  (0,\infty) \to [0,\infty)$ such that, for some $\nu > 0$,
  \begin{equation*}
    \int_0^\infty (f_0'(y))^2 \, y^4\, e^{\nu y} \,dy < \infty.
  \end{equation*}
  Let $f$ be the solution of Smoluchowski's equation
  \eqref{eq:coag-eq} with initial condition $f_0$ (this problem is
  known to be well-posed; see
  \cite{citeulike:1874300,citeulike:1103441}). Then there is a number
  $\mu > 0$ such that for any $0 < \delta < 1$,
  \begin{equation}
    \label{eq:H1 convergence - intro}
    \int_0^\infty \left(
      t^2 \partial_y (f(t, t y)) - \partial_y g_\rho
    \right)^2
    \,y^4 e^{\mu y}\,dy
    \leq
    K^2 t^{-2 \delta}
    \quad
    \text{ for all } t > 0,
  \end{equation}
  for some $K > 0$ which depends only on the initial condition $f_0$.
  In particular,
  \begin{equation*}
    \sup_{y > 0} \left\{
      y^2 \abs{
        t^2 f(t, t y) - g_\rho(y)
      }
    \right\}
    \leq
    K t^{-\delta}
    \quad
    \text{ for all } t > 0,
  \end{equation*}
  where $\rho$ is the mass of $f_0$ and $g_\rho(y) := \frac{4}{\rho}
  e^{-\frac{2}{\rho}y}$ for $y > 0$.
\end{thm}
This theorem is a direct consequence of Theorem
\ref{thm:global-convergence-intro} below, proved in Section
\ref{sec:global}. Results along these lines were already obtained in
\cite{citeulike:1103441} and \cite{MR2139564}, where it was proved
that with some regularity of the initial condition it holds that
\begin{equation}
  \label{eq:MP - uniform convergence}
  \sup_{y > 0} \left\{
    y \abs{
      t^2 f(t, t y) - g_\rho(y)
    }
  \right\}
  \to
  0
  \quad
  \text{ when } t \to \infty.
\end{equation}
A result similar to the latter, but giving uniform convergence only in
compact sets, can be found in \cite{citeulike:1103487}; weak
convergence was proved in~\cite{citeulike:1069613} using entropy
arguments. Hence, the main improvement of Theorem \ref{thm:uniform
  convergence - intro} is that we give an explicit rate of
convergence, which is to our knowledge a new result, and that the
convergence in eq.  \eqref{eq:H1 convergence - intro} is stronger than
uniform convergence. In fact, we give corresponding results for the
rate of convergence of higher derivatives of $f$ as long as the
initial condition is regular enough. We note that this rate is optimal
in general, as can be seen from the evolution of the moment of order
$0$ of $f$, $\int f$. These results, and the methods we use to arrive
at them, are best described in terms of a rescaled form of eq.
(\ref{eq:coag-eq}), which we introduce next.

\subsection{The self-similar equation}
\label{sec:self-similar-eq}

A \emph{scaling solution} or \emph{self similar solution} of
Smoluchowski's equation is a solution $f$ to eq. (\ref{eq:coag-eq})
which is a rescaling of some fixed function $g$ for any time $t$:
\begin{equation}
  \label{eq:def-scaling}
  f(t,y) = p(t) g(q(t) y)
  \qquad (t, y > 0),
\end{equation}
where $p, q > 0$ are functions of time. When such a solution exists,
$g$ is called a \emph{scaling profile} or \emph{self-similar
  profile}.

Let us briefly describe the problem of the scaling behavior of
Smoluchowski's equation for a general coagulation rate, and then we
will focus on the problem for a constant rate. It is expected that for
a general class of homogeneous coagulation kernels, there is a unique
self-similar profile with given mass $\rho > 0$, and that for a very
wide class of initial conditions solutions to Smoluchowski's equation
approximate, for large times, the self-similar solution with the same
mass, in a sense to be precised. For general homogeneous coagulation
rates there are no rigorous proofs of this behavior, except for recent
results that have shown the existence of self-similar profiles
\cite{citeulike:1172261,MR2114413} and given some of their properties
\cite{citeulike:1271620,citeulike:1069567}.  Nevertheless, for the
specific coagulation rates given by $a(x,y) = x + y$ and $a(x,y) = 1$
(our case), it is known that the self-similar profile is unique for
each given finite mass, and the convergence to it has been proved
in~\cite{citeulike:1103487, citeulike:1103441,
  MR2139564,citeulike:1069613}. For $a(x,y) = 1$, the profile
corresponding to a mass $\rho$ is the function $g_\rho$ mentioned in
Theorem \ref{thm:uniform convergence - intro}, and the convergence to
it holds at least in the sense of eq.  (\ref{eq:MP - uniform
  convergence}) under quite general conditions on the initial data.

In our case, if $f$ is a self-similar solution of eq.
(\ref{eq:coag-eq}), then it is known that the scaling in eq.
(\ref{eq:def-scaling}) is determined: it must happen that for some
$\tau \in \RR$,
\begin{equation*}
  f(t+\tau, y)
  =
  \frac{1}{(1+t)^2}\,
  g \Big( \frac{y}{1+t} \Big)
  \quad \text{ for } t \geq 0, y > 0.
\end{equation*}
Hence, up to time translations, we can consider that all self-similar
solutions are of this form. The following change of variables is then
useful: if $f$ is a solution of eq. (\ref{eq:coag-eq}) for $t \in
[0,\infty)$, we set
\begin{equation}
  \label{eq:change-forward}
  g(t,y) := e^{2t} f\big( e^{t} - 1,\, e^t y \big)
  \quad \text{ for } t \geq 0, y > 0,
\end{equation}
whose inverse is
\begin{equation}
  \label{eq:change-backward}
  f(t,y) = \frac{1}{(1+t)^2} \,
  g \Big( \log (1+t),\,
  \frac{y}{1+t} \Big)
  \quad \text{ for } t \geq 0, y > 0.
\end{equation}
The fact that $f$ is a solution of eq. (\ref{eq:coag-eq}) is
equivalent to $g$ satisfying the following \emph{self-similar
  coagulation equation}:
\begin{align}
  \label{eq:ss-coag-lambda=0}
  \partial_t g
  = & 2g + y \partial_y g
  + C(g, g)
  \\
  \label{eq:ss-coag-lambda=0-operator}
  =: & D(g) +  C(g, g).
\end{align}
Through this change of variables, self-similar solutions $f$ of eq.
(\ref{eq:coag-eq}) correspond, up to a time translation, to stationary
solutions of eq.~\eqref{eq:ss-coag-lambda=0}. Given this equivalence,
we study eq.~\eqref{eq:ss-coag-lambda=0}, which is more convenient
when considering the scaling behavior of solutions.

Let us be more precise as to what we mean by a solution to
eq.~(\ref{eq:ss-coag-lambda=0}). Below we denote our functional spaces
by
\begin{equation*}
  L^1 := L^1(0,\infty)
  \qquad
  L^1_1 := L^1((0,\infty); y\,dy).
\end{equation*}
%
\begin{dfn}
  For any $g, h \in L^1$, we define $C(g,h)$ by
  \begin{equation}
    \label{eq:dfn-bilinear-C}
    \begin{split}
      C(g,h)(y) := &\frac{1}{2} \int_0^y g(y') h(y-y') \,dy
      \,dy'
      \\
      & - \frac{1}{2} \int_0^\infty g(y) h(y') \,dy'
      \\
      & - \frac{1}{2} \int_0^\infty h(y) g(y') \,dy'.
    \end{split}
  \end{equation}
  Notice that this is in agreement with the expression in eq.
  (\ref{eq:coag_operator}), and that this operator is symmetric:
  $C(g,h) = C(h,g)$ for any $g$, $h$.
\end{dfn}
\begin{dfn}
  A \emph{solution} of eq. (\ref{eq:ss-coag-lambda=0}) on $[0,\infty)$
  is a nonnegative function $g \in \mathcal{C}([0,\infty); L^1 \cap
  L^1_1)$ such that eq. (\ref{eq:ss-coag-lambda=0}) holds in the
  following sense: for all $\phi \in \mathcal{C}_0^\infty$,
  $g$ satisfies
  \begin{multline}
    \label{eq:def-solution}
    \int_0^\infty g(t,y) \phi(y) \,dy
    =
    \int_0^\infty g(0,y) \phi(y) \,dy
    +
    \int_0^t \int_0^\infty g(s,y) \phi(y) \,dy \,ds
    \\
    - \int_0^t \int_0^\infty y\,  g(s,y) \phi'(y) \,dy\,ds
    \\
    + \int_0^t \int_0^\infty C(g(s),g(s))(y) \phi(y) \,dy\,ds
    \quad
    \text{ for } t > 0.
  \end{multline}
\end{dfn}
The results in \cite{citeulike:1069567} or \cite{citeulike:1103441}
prove that for any nonnegative $g^0 \in L^1 \cap L^1_1$ there exists a
unique solution to eq. (\ref{eq:ss-coag-lambda=0}) in the above sense
with the initial condition $g(0,y) = g^0(y)$ for $y > 0$.

Eq. \eqref{eq:ss-coag-lambda=0} has a one-parameter family of
stationary solutions $\{g_\rho\}_{\rho > 0}$ which is explicitly
given by
\begin{equation}
  \label{eq:stationary-sols-lambda=0}
  g_\rho(y) :=
  \frac{4}{\rho} e^{-\frac{2y}{\rho}}.
\end{equation}
The solution $g_\rho$ has mass (first moment) $\rho$ and number of
particles (zeroth moment) equal to 2. If we write the result in eq.
(\ref{eq:MP - uniform convergence}) in terms of a solution $g$ to eq.
(\ref{eq:ss-coag-lambda=0}), then it means that
\begin{equation*}
  \lim_{t \to \infty}
  \sup_{y > 0}
  \left\{
    y\abs{g(t,y) - g_\rho(y)}
  \right\}
  = 0
  .
\end{equation*}
Some regularity conditions need to be imposed on the initial data for
the latter convergence to hold; we refer the reader to
\cite{MR2139564}, where the result is proved, for details.

\subsection{Main results}
\label{sec:main-results}

Our main objective is to give an explicit exponential rate of
convergence to equilibrium for solutions to the above equation in
various norms. Explicitly, for an integrable function $h:(0,\infty)
\to \RR$, a positive number $\mu$ and an integer $k \geq -1$ we
consider the norms
\begin{equation}
  \label{eq:(k,mu) norm-intro}
  \norm{h}_{k,\mu}^2
  := \int_0^\infty y^{2(k+1)} (D^k h(y))^2 e^{\mu y} \,dy,
\end{equation}
where $D^k h$ represents the $k$-th derivative of $h$ and $D^{-1} h$
denotes the primitive of $h$ given by
\begin{equation}
  \label{eq:dfn-H-intro}
  D^{-1} h (y) \equiv -H(y) := -\int_y^\infty h(x) \,dx.
\end{equation}
Note that the subindex $\mu$ indicates the exponential weight, and not
the power of the function as is usual when considering $L^p$ spaces.
The latter norm with $\mu = 2/\rho$ is directly suggested by the
quadratic approximation to the following \emph{relative entropy}
functional near a stationary solution, and is the initial motivation
for it:
\begin{equation}
  \label{eq:relative-entropy}
  F[g \vert g_\rho] :=
  \int_0^\infty \left(
    g(y)
    \left( \log\frac{g(y)}{g_\rho(y)} - 1 \right)
    + g_\rho(y)
  \right) \,dy.
\end{equation}

We prove both local and global convergence results; among the local
ones we have the following, proved in Section \ref{sec:local}:
\begin{thm}
  \label{thm:local exponential convergence-intro}
  Let $g$ be a solution to eq. \eqref{eq:ss-coag-lambda=0} with mass
  $\rho$ and nonnegative initial condition $g^0$ at $t=0$. Take $0 <
  \mu < \nu < 2/\rho$ and an integer $k \geq -1$. Choose $\epsilon >
  0$, and assume that
  \begin{enumerate}
  \item The exponential moment of order $\nu/2$ is bounded by a
    constant $E > 0$ for all $t \geq 0$; this is,
    \begin{equation*}
      \int_0^\infty \abs{g(t,y)} e^{\frac{\nu}{2} y} \,dy
      \leq
      E
      \quad \text{ for all } t \geq 0.
    \end{equation*}
  \item \label{eq:entropy condition}
    The relative entropy $F[g^0 | g_\rho]$ is less than $\epsilon$.
  \end{enumerate}
  Then, for any $0 < \delta < 1$ there is an $\epsilon_0 =
  \epsilon_0(\rho, E, \delta, k, \mu, \nu)$ such that if $\epsilon
  \leq \epsilon_0$ then
  \begin{equation}
    \label{eq:local-convergence-(k,mu)}
    \norm{g(t)-g_\rho}_{k,\mu}
    \leq
    K \norm{g^0 - g_\rho}_{k,\nu} e^{-\delta t}
    \quad
    \text{ for } t \geq 0,
  \end{equation}
  for some constant $K = K(\rho, E, \delta, k, \mu, \nu)$. When
  $k=-1$, the same is true with $\norm{g^0 -
    g_\rho}_{-1,\mu}$ instead of $\norm{g^0 - g_\rho}_{-1,\nu}$:
  \begin{equation*}
    \norm{g(t)-g_\rho}_{-1,\mu}
    \leq
    K \norm{g^0 - g_\rho}_{-1,\mu} e^{-\delta t}
    \quad
    \text{ for } t \geq 0.
  \end{equation*}
  Observe that these results are meaningful when $\norm{g^0}_{k,\nu}$,
  or $\norm{g^0}_{k,\mu}$ in this last case, are finite.
\end{thm}

\begin{rem}
  Below we give intrinsic conditions on the initial data which are
  equivalent to point $1$ in the above theorem (see Section
  \ref{sec:behavior-moments}, and concretely eq.
  \eqref{eq:condition-exp-moment-uniformly-bounded}), so it is really
  a condition on $g^0$. In particular, it is satisfied if $g^0(y) \leq
  \nu e^{-\frac{\nu}{2} y}$ for all $y > 0$.
\end{rem}

Our main global convergence result can be stated as follows:

\begin{thm}
  \label{thm:global-convergence-intro}
  Let $g$ be a solution to eq. \eqref{eq:ss-coag-lambda=0} with mass
  $\rho$ and nonnegative initial condition $g^0$ such that $(g^0)' \in
  L^1$. Assume that, for some $\nu > 0$ and $k \geq -1$,
  \begin{equation*}
    \int_0^\infty g^0(y) e^{\nu y} \,dy < \infty
    \quad \text{ and }
    \quad
    \norm{g^0}_{k,\nu} < \infty.
  \end{equation*}
  Then there exists $0 < \mu < 2/\rho$ such that for any $0 < \delta <
  1$,
  \begin{equation}
    \label{eq:global-convergence}
    \norm{g(t)-g_\rho}_{k,\mu}
    \leq
    K e^{-\delta t}
    \quad
    \text{ for } t \geq 0
  \end{equation}
  for some number $K$ which depends in an explicit way on $k, \nu,
  \delta$ and $g^0$.
\end{thm}
For $k = 1$, this implies Theorem \ref{thm:uniform convergence -
  intro} through the self-similar change of variables in eq.
(\ref{eq:change-forward}). When $k = -1$ or $0$, the condition
$\norm{g^0}_{k,\nu} < \infty$ is redundant, as it is already implied
by $\int g^0 e^{\nu y} < \infty$ and $g_0' \in L^1$.  We also
remark that all the constants in the previous theorems are
constructive and can be explicitly given by following their proofs.


  We point out that all the arguments in the paper can be in fact
  carried out with the norms
  \begin{equation}
    \label{eq:(k,mu) norm-alternative}
    \tnorm{h}_{k,\mu}^2
    := \int_0^\infty y^{2k} (D^k h(y))^2 e^{\mu y} \,dy
  \end{equation}
  for $k \geq 0$ and $\mu > 0$, and
  \begin{equation}
    \label{eq:(-1,mu)_norm-alternative}
    \tnorm{h}_{-1,\mu}^2 = \norm{h}_{-1,\mu}^2
    := \int_0^\infty H(y)^2 e^{\mu y} \,dy,
  \end{equation}
  as before. Notice that here the power weight is lower in all norms
  with $k \geq 0$. In this case, Theorems \ref{thm:local exponential
    convergence-intro} and \ref{thm:global-convergence-intro} are
  still true with the sole difference that one obtains an exponential
  rate of convergence with $\delta < 1/2$ instead of $\delta < 1$. In
  particular, one obtains a statement as that in Theorem
  \ref{thm:uniform convergence - intro} with a modified power weight
  and rate of convergence: in the conditions of Theorem
  \ref{thm:uniform convergence - intro}, and for $0 < \delta < 1/2$,
  \begin{equation*}
    \int_0^\infty \left(
      t^2 \partial_y (f(t, t y)) - \partial_y g_\rho
    \right)^2
    \,y^2 e^{\mu y}\,dy
    \leq
    K^2 t^{-2 \delta}
    \quad
    \text{ for all } t > 0,
  \end{equation*}
  and in particular,
  \begin{equation*}
    \sup_{y > 0} \left\{
      y \abs{
        t^2 f(t, t y) - g_\rho(y)
      }
    \right\}
    \leq
    K t^{-\delta}
    \quad
    \text{ for all } t > 0.
  \end{equation*}
  We have not included the calculations for the norms \eqref{eq:(k,mu)
    norm-alternative} with different power weights in order to make
  the paper more readable, but there are no difficulties in rederiving
  our results in this case.
  
  We also note that the result on exponential convergence in the $L^2$
  norm, with an exponential rate of convergence of $1/2$, can be
  arrived at by explicit computations on the solution of the Fourier
  transform of eq. \eqref{eq:coag-eq}. We do this in Lemma \ref{lem:L2
    explicit convergence}.

  These results strongly suggest that the behavior near $y=0$ of
  solutions to eq. \eqref{eq:coag-eq} is decisive when studying the
  speed of convergence of solutions to a self-similar one: it seems
  likely that one cannot obtain, for example, the convergence of
  derivatives at $y=0$, so some weight near $y=0$ is probably
  unavoidable.


The strategy of proof for our results stems from the relative
entropy functional defined in eq.  \eqref{eq:relative-entropy},
which was first introduced in the context of eq.
(\ref{eq:ss-coag-lambda=0}) in \cite{citeulike:1069613}, where it
was proved that
\begin{gather}
  \label{eq:entropy_g_decreasing}
  \frac{d}{dt} \frac{F[g(t,\cdot) \vert g_\rho]}{\int g(t,y)\,dy}
  \leq
  0,
  \quad \text{ and }
  \\
  \label{eq:entropy_G_decreasing}
  \frac{d}{dt} F[ G(t,\cdot) \vert G_\rho ]
  \leq
  0,
\end{gather}
for any solution $g$ of eq. \eqref{eq:ss-coag-lambda=0}, where $G$ and
$G_\rho$ are the primitives of $g$ and $g_\rho$ defined as in eq.
(\ref{eq:dfn-H-intro}). That is: the above expressions are Liapunov
functionals for this equation. If one wants to carry out the classical
method of studying the convergence to equilibrium in a neighborhood of
a stationary solution by looking at the spectral properties of the
linearization of the equation near such a solution, the above
functionals suggest that we consider the linearized operator in the
norm which appears in their quadratic approximation. For
eq. (\ref{eq:entropy_g_decreasing}), this norm is precisely
$\norm{\cdot}_{0,\frac{2}{\rho}}$; for the functional in
eq. (\ref{eq:entropy_G_decreasing}), it is
$\norm{\cdot}_{-1,\frac{2}{\rho}}$. We follow this path and study the
equation
\begin{equation}
  \label{eq:linear-ss0}
  \partial_t h = Lh := 2h + y \partial_y h + 2 C(h,g_\rho),
\end{equation}
where $L$ is the linearization of the operator $D(g) + C(g,g)$ from
eq. (\ref{eq:ss-coag-lambda=0-operator}) near the stationary solution
$g_\rho$. We find that $L$, when restricted to functions $h$ such that
$\int_0^\infty y h(y) \,dy = 0$, has a spectral gap in the norm
$\norm{\cdot}_{k,\mu}$ for $0 < \mu \leq 2/\rho$ and integer $k \geq
-1$; see Sections \ref{sec:gap-H-1}--\ref{sec:spectral-gap-derivative}
for a precise statement (see also \cite{EM08} for a spectral study of
$L$ in an $L^1$ framework). This shows exponential convergence to $0$
for the linear equation (\ref{eq:linear-ss0}) in these norms, and can
also be used to prove exponential convergence to equilibrium for the
nonlinear equation (\ref{eq:ss-coag-lambda=0}) in a
close-to-equilibrium regime.

In order to treat solutions far from the equilibrium, the results of
\cite{MR2139564} were already available. However, as we are interested
in the rate of convergence, we prove a similar result which gives an
explicit rate of convergence in the $L^2$ norm, using estimates on the
explicit solution for the Fourier transform of
\eqref{eq:ss-coag-lambda=0} which resemble those in the proof of
uniform convergence in \cite{MR2139564}. With further estimates on the
derivatives of the explicit solution one could probably reproduce ``by
hand'' our results on the rate of convergence for the derivatives.
Nevertheless, the calculations soon become very involved, and in our
opinion do not add much to the understanding of the equation, while a
linear study proves to be less technical, and possibly better suited
for generalization to other coagulation kernels.

The paper is organized as follows: in Section \ref{sec:Functional
  inequalities} we derive our main functional inequality, obtained
from a linearization of an inequality of Aizenman and Bak
\cite{citeulike:1081959}, which is then used to prove that the linear
operator $L$ has a spectral gap in the norm
$\norm{\cdot}_{-1,2/\rho}$. In Section \ref{sec:spectral gaps} we
define precisely the linear operator $L$ in the spaces associated to
the norms mentioned above, and prove that it has a spectral gap in
them. Section \ref{sec:behavior-moments} can be read independently
from the rest of the paper; in it, we find an explicit expression for
the evolution of the exponential moments of solutions to eq.
(\ref{eq:ss-coag-lambda=0}). This is a remarkable piece of
information, which can be obtained because of the widely known fact
that the Laplace transform of eq. (\ref{eq:coag-eq}) is explicitly
solvable. In turn, this information on the exponential moments is
needed in the arguments of later sections. In Section \ref{sec:local}
we give local exponential convergence results for the nonlinear
equation (\ref{eq:ss-coag-lambda=0}) with the help of the linear study
from Section \ref{sec:spectral gaps}, and in particular we prove
Theorem \ref{thm:local exponential convergence-intro}. Finally, in
Section \ref{sec:global} we state and prove our global convergence
results, and in particular prove Theorem
\ref{thm:global-convergence-intro}.

\section{Functional inequalities}
\label{sec:Functional inequalities}

The study of the linearization of eq. (\ref{eq:ss-coag-lambda=0}) in
the norm $\norm{\cdot}_{-1,2/\rho}$ defined in eq.  \eqref{eq:(k,mu)
  norm-intro} is directly suggested by the Liapunov functional in eq.
(\ref{eq:entropy_G_decreasing}). Here we prove the necessary
inequality to show that the linearized operator $L$ (which we define
more precisely later) has a spectral gap in the norm
$\norm{\cdot}_{-1,2/\rho}$; more precisely,
\begin{equation}
  \label{eq:spectral-gap-L-mentioned}
  \ap{h, Lh}_{-1,\frac{2}{\rho}}
  \leq - \norm{h}^2_{-1,\frac{2}{\rho}}
\end{equation}
for any sufficiently regular function $h$ with $\int_0^\infty y h(y)
\,dy = 0$, where $\ap{\cdot,\cdot}_{-1,2/\rho}$ denotes the scalar
product associated to the norm $\norm{\cdot}_{-1,2/\rho}$; see
Proposition \ref{prp:spectral-gap} and Lemma \ref{lem:spectral gap
  X_-1,2/rho} for a more precise statement.

In order to prove this we use an inequality of Aizenman and Bak, which
was first introduced in their paper \cite{citeulike:1081959} to study
the convergence to equilibrium for eq. (\ref{eq:coag-eq}) with an
additional fragmentation term.  In \cite{citeulike:1069613},
Aizenman-Bak's inequality is also needed in the proof that the
relative entropy $F[G \vert G_\rho]$ is nonincreasing, and here we use
a quadratic approximation of it to show eq.
(\ref{eq:spectral-gap-L-mentioned}).

We will first prove a needed auxiliary inequality; then, we prove
eq. (\ref{eq:spectral-gap-L-mentioned}).

\subsection{Some notation}
\label{sec:notation}

Here we explain some of our notation conventions, and in particular
that for weighted $L^p$ spaces; for the specific spaces used in our
study of the linearized coagulation operator, see Section
\ref{sec:functional spaces}.

As most of our functions will be defined on $(0,\infty)$, we omit the
limits in integrals when this does not lead to confusion; thus, when
we omit the limits of integration, it is understood that integration
is on $(0,\infty)$; when we omit the variable, it is understood that
integration is on the $y$ variable. For example, $\int_0^\infty g(y)
\,dy$ is understood to be an integral over $y \in (0,\infty)$, and for
a solution $g = g(t,y)$ depending on time $t$ and size $y$, the
expression $\int g$ denotes $\int_0^\infty g(t,y) \,dy$. In the same
way, derivatives are with respect to the size variable unless
explicitly noted.

We use several spaces of functions on $(0,\infty)$, and we similarly
omit the explicit mention of the interval. Thus, we use:
\begin{equation}
  \label{eq:def-L1-L11}
  L^1 := L^1(0,\infty)
  \qquad
  L^1_k := L^1((0,\infty); y^k\,dy)
  \quad (k \in \RR).
\end{equation}
In general, for a function $y \mapsto m(y)$ and $1 \leq p < \infty$,
we write
\begin{equation}
  \label{eq:def-L1-relative}
  L^p(m(y)) := L^p((0,\infty); m(y)\,dy),
\end{equation}
to denote the space of functions $f$ such that $y \mapsto \abs{f(y)}^p
m(y)$ is integrable. Likewise, $\mathcal{C}$ and $\mathcal{C}^k$
denote the spaces of continuous and $k$-times continuously
differentiable functions on $(0,\infty)$, respectively;
$\mathcal{C}_0$, $\mathcal{C}^k_0$ are the corresponding spaces of
compactly supported functions on $(0,\infty)$.

For any function which is denoted $g$ or $h$, and which is integrable
on $(\epsilon, \infty)$ for any $\epsilon > 0$, we always write $G$,
$H$ to mean the primitives given by
\begin{equation}
  \label{eq:def-G-notation}
  G(y) := \int_y^\infty g(x) \,dx,
  \quad
  H(y) := \int_y^\infty h(x) \,dx
  \qquad (y > 0).
\end{equation}
The same is done for the pair $g_\rho$, $G_\rho$, where $g_\rho$ is
the stationary solution from eq. (\ref{eq:stationary-sols-lambda=0}).
Hence,
\begin{equation}
  \label{eq:expression_G_rho-notation}
  G_\rho(y)
  :=
  \int_y^\infty g_\rho(x) \,dx
  =
  2 e^{-\frac{2y}{\rho}}
  \qquad (y > 0)
  .
\end{equation}

\subsection{Aizenman-Bak's inequality and its linea\-ri\-za\-tion}

The following lemma can be found in \cite[Proposition
4.3]{citeulike:1081959}.

\begin{lem}[Aizenman-Bak inequality]
  \label{lem:Aizenman-Bak}
  Take $f: (0,\infty) \to \RR$ with $f \geq 0$, $f \in L^1$, and $f
  \log f \in L^1$. Then,
  \begin{multline}
    \label{eq:Aizenman-Bak}
    \int_0^\infty\!\!\! \int_0^\infty f(x) f(y) \log f(x+y) \,dx \,dy
    \\
    \leq
    \int_0^\infty\!\!\! \int_0^\infty f(x) f(y) \log f(y) \,dx\,dy
    - \left(\int_0^\infty f(x) \,dx \right)^2,
  \end{multline}
  and the equality is attained only for the exponential functions
  $f(y) = e^{-\mu y}$, with $\mu > 0$.
\end{lem}

Now we expand this inequality to second order near an exponential
function to find the following:

\begin{lem}[Linearized Aizenman-Bak inequality]
  \label{lem:Linear-Aizenman-Bak}
  For any function $h \in L^2((1+y) e^{\mu y})$ and
  any $\mu > 0$, it holds that
  \begin{equation}
    \label{eq:Linear-Aizenman-Bak}
    4 \int_0^\infty h H e^{\mu y}
    \leq
    2\mu \left(\int h\right) \left(\int yh\right)
    + \int_0^\infty h^2 y e^{\mu y}
    + \frac{1}{\mu} \int_0^\infty h^2 e^{\mu y},
  \end{equation}
  where $H$ was given in (\ref{eq:def-G-notation}).
\end{lem}


\begin{proof}
  Fix $\mu > 0$, and take a continuous function $h$ with compact
  support on $(0,\infty)$. For small enough $\epsilon > 0$, the
  function
  \begin{equation}
    \label{eq:def-f_epsilon}
    f_\epsilon(y) := e^{-\mu y} + \epsilon h(y)
    \qquad (y > 0)
  \end{equation}
  is in the conditions of Lemma~\ref{lem:Aizenman-Bak}, so
  Aizenman-Bak's inequality holds for the function $f_\epsilon$, for
  all small enough $\epsilon$. One can easily check that both sides of
  the inequality are differentiable in $\epsilon$. As the equality is
  attained at $\epsilon = 0$, it must hold that
  \begin{equation}
    \label{eq:dd-0}
    \frac{d^2}{d\epsilon^2} \Big |_{\epsilon=0} I_1
    \leq
    \frac{d^2}{d\epsilon^2} \Big |_{\epsilon=0} I_2,
  \end{equation}
  where $I_1$ and $I_2$ are the left and right hand sides of
  Aizenman-Bak's inequality for $f_\epsilon$, respectively. This will
  give the inequality we are interested in, so let us calculate these
  terms.

  \paragraph{Step 1: First order derivative.}

  Let us calculate the first order derivative of all terms in the
  inequality. The regularity of $f_\epsilon$ justifies the derivation
  under the integral sign, and we obtain the following for small
  enough $\epsilon$:
  \begin{multline}
    \label{eq:d-1}
    \frac{d}{d\epsilon} I_1
    =
    \frac{d}{d\epsilon}
    \int_0^\infty\!\!\! \int_0^\infty
    f_\epsilon(x) f_\epsilon(y) \log f_\epsilon(x+y) \,dx \,dy
    \\
    =
    2 \int_0^\infty\!\!\! \int_0^\infty
    h(x) f_\epsilon(y) \log f_\epsilon(x+y) \,dx\,dy
    \\
    +
    \int_0^\infty\!\!\! \int_0^\infty
    f_\epsilon(x) f_\epsilon(y) \frac{h(x+y)}{f_\epsilon(x+y)}
    \,dx\,dy
    =: I_{11} + I_{12}.
  \end{multline}
  For the first term in $I_2$,
  \begin{multline}
    \label{eq:d-2}
    \frac{d}{d\epsilon}
    \int_0^\infty\!\!\! \int_0^\infty
    f_\epsilon(x) f_\epsilon(y) \log f_\epsilon(y) \,dx\,dy
    \\
    =
    \int_0^\infty h(x)
    \int_0^\infty f_\epsilon(y) \log f_\epsilon(y) \,dy \,dx
    \\
    +
    \int_0^\infty f_\epsilon(x)
    \int_0^\infty h(y) \log f_\epsilon(y) \,dy \,dx
    \\
    +
    \left( \int h \right) \left( \int f_\epsilon \right)
    =:
    I_{21} + I_{22} + I_{23}.
  \end{multline}
  And for the last term in $I_2$ we have
  \begin{equation}
    \label{eq:d-3}
    - \frac{d}{d\epsilon}
    \left(\int_0^\infty f_\epsilon(x) \,dx \right)^2
    =
    - 2  \left( \int f_\epsilon \right)  \left( \int h \right)
    =: I_{24}.
  \end{equation}

  \paragraph{Step 2: Second order derivative.}

  We have:
  \begin{multline}
    \label{eq:dd-1}
    \frac{d}{d\epsilon} \Big |_{\epsilon = 0} I_{11}
    =
    -2 \mu \int_0^\infty\!\!\! \int_0^\infty
    h(x) h(y) (x+y) \,dx\,dy
    \\
    +
    2 \int_0^\infty\!\!\! \int_0^\infty
    h(x) e^{\mu x} h(x+y)
    \,dx \,dy
    \\
    =
    -4 \mu \left( \int h \right) \left( \int y h \right)
    + 2 \int h H e^{\mu y}
  \end{multline}

  \begin{multline}
    \label{eq:dd-2}
    \frac{d}{d\epsilon} \Big |_{\epsilon = 0} I_{12}
    =
    2 \int_0^\infty\!\!\! \int_0^\infty
    h(x) e^{\mu x} h(x+y)
    \,dx \,dy
    \\
    -
    \int_0^\infty\!\!\! \int_0^\infty
    e^{\mu(x+y)} h(x+y)^2
    \,dx \,dy
    \\
    =
    2 \int h H e^{\mu y}
    - \int h^2 \, y e^{\mu y}.
  \end{multline}

  \begin{equation}
    \label{eq:dd-3}
    \frac{d}{d\epsilon} \Big |_{\epsilon = 0} I_{21}
    =
    - \mu \left( \int h \right) \left( \int y h \right)
    + \left( \int h \right)^2.
  \end{equation}

  \begin{equation}
    \label{eq:dd-4}
    \frac{d}{d\epsilon} \Big |_{\epsilon = 0} I_{22}
    =
    - \mu \left( \int h \right) \left( \int y h \right)
    + \frac{1}{\mu} \int h^2 e^{\mu y}.
  \end{equation}

  \begin{equation}
    \label{eq:dd-5}
     \frac{d}{d\epsilon} \Big |_{\epsilon = 0} I_{23}
     = \left( \int h \right)^2.     
  \end{equation}

  \begin{equation}
    \label{eq:dd-6}
    \frac{d}{d\epsilon} \Big |_{\epsilon = 0} I_{24}
    = -2 \left( \int h \right)^2.     
  \end{equation}

  \paragraph{Step 3: Final inequality.}

  Putting together eqs. \eqref{eq:dd-1} and \eqref{eq:dd-2} we have
  \begin{equation}
    \label{eq:dd-I_1}
    \frac{d^2}{d\epsilon^2} \Big |_{\epsilon=0} I_1
    =
    4 \int h H e^{\mu y}
    - 4 \mu \left( \int h \right) \left( \int y h \right)
    - \int h^2 \, y e^{\mu y},
  \end{equation}
  and eqs. \eqref{eq:dd-3}--\eqref{eq:dd-6} show that
  \begin{equation}
    \label{eq:dd-I_2}
    \frac{d^2}{d\epsilon^2} \Big |_{\epsilon=0} I_2
    =
    - 2 \mu \left( \int h \right) \left( \int y h \right)
    + \frac{1}{\mu} \int h^2 e^{\mu y}.
  \end{equation}

  Finally, writing out the inequality in eq. \eqref{eq:dd-0} in view
  of eqs. \eqref{eq:dd-I_1}--\eqref{eq:dd-I_2}, we obtain precisely
  eq. \eqref{eq:Linear-Aizenman-Bak}. This shows our inequality for
  any $h \in \mathcal{C}_0(0,\infty)$, and taking a suitable limit of
  such functions shows the inequality for any $h \in L^2((1+y) e^{\mu
    y})$ (note that for such an $h$ all terms in the inequality are
  finite).
\end{proof}

\subsection{Spectral gap inequality}

Take $\rho > 0$, which will be kept fixed in the rest of this section.
One can linearize the self-similar eq. \eqref{eq:ss-coag-lambda=0}
around the profile $g_\rho$ with mass $\rho$ (given in
eq. (\ref{eq:stationary-sols-lambda=0})) to obtain $\partial_t h =
Lh$, with $L$ given by the following definition:
\begin{dfn}
  \label{dfn:L-prev}
  For a function $h \in \mathcal{C}^1_0$, we set
  \begin{equation}
    \label{eq:dfn:L}
    L h(y) := 2 h(y) + y h'(y) + 2 C(g_\rho, h)(y).
  \end{equation}
\end{dfn}
Using the expression of $C(g_\rho,h)$ from eq.
\eqref{eq:dfn-bilinear-C} and noticing that $h * g_\rho =
-\frac{4}{\rho} H + g_\rho \int h + \frac{2}{\rho} H * g_\rho$ (an
integration by parts), one has
\begin{equation}
  \label{eq:simpler-C}
  2 C(g_\rho, h) = h * g_\rho - g_\rho \int h - 2 h
  \\
  =
  -2h -\frac{4}{\rho} H + \frac{2}{\rho} H * g_\rho,
\end{equation}
so \eqref{eq:dfn:L} may be rewritten as
\begin{equation}
  \label{eq:simpler-L}
  Lh =
  y \partial_y h  - \frac{4}{\rho} H + \frac{2}{\rho} H * g_\rho.
\end{equation}
Observe that this expression can be generalized to functions $h$ which
are not necessarily integrable at $y=0$, as it is $H$ which appears in
the convolution; it is for this reason that it will be useful later.
The following expression for the primitive of the operator $C$ will
also be needed below:
\begin{lem}
  \label{lem:primitive_C}
  For $g, h \in L^1$, and $y > 0$, it holds that
  \begin{equation*}
    \int_y^\infty C(g,h)(x) \,dx
    =
    \frac{1}{2}
    \int_0^y g(x) H(y-x) \,dx
    -
    \frac{1}{2}
    H(y) \int_0^\infty g(x)\,dx,
  \end{equation*}
  or, written more compactly,
  \begin{equation*}
    2 \int_y^\infty C(g,h)
    = g * H - H \int g,
  \end{equation*}
  where $H$ is given in eq. (\ref{eq:def-G-notation}).
\end{lem}
A proof of the above result can be obtained by direct integration of
eq. (\ref{eq:dfn-bilinear-C}), and we omit it here. For an explicit
proof, see for example that of Proposition 9 in
\cite{citeulike:1069613}. A direct consequence of this is the next
expression for the primitive of the operator $L$:


\begin{lem}
  \label{lem:primitive-L}
  For $h \in \mathcal{C}^1_0$ and $y > 0$, the primitive of $L h$ is
  given by
  \begin{equation}
    \label{eq:primitive-L}
    \mathcal{L} h(y)
    :=
    \int_y^\infty Lh(x) \,dx
    = -H(y) - y h(y) + g_\rho * H(y)
    ,
  \end{equation}
  where $H$, $G_\rho$ are the primitives of $h$, $g_\rho$ given in
  eqs.  (\ref{eq:def-G-notation}) and
  (\ref{eq:expression_G_rho-notation}), respectively.
\end{lem}

\begin{proof}
  Using Lemma \ref{lem:primitive_C} and integrating by parts,
  \begin{multline*}
    \mathcal{L} h (y)
    =
    2 H(y) + \int_y^\infty x\,h'(x) \,dx
    + 2 \int_y^\infty C(h,g_\rho)(x) \,dx
    \\
    =
    H(y) - y h(y) + g_\rho * H - H \int g_\rho
    =
    - H(y) - y h(y) + g_\rho * H,
  \end{multline*}
  as the integral of $g_\rho$ is equal to $2$.
\end{proof}

\begin{prp}[Spectral gap inequality]
  \label{prp:spectral-gap}
  For any $h \in \mathcal{C}^1_0$ such that
  \begin{equation*}
    \int_0^\infty H(y)^2\,
    e^{\frac{2}{\rho} y} \,dy < \infty
    \quad \text{ and } \quad
    \int_0^\infty y\,h(y) \,dy = 0,
  \end{equation*}
  it holds that
  \begin{equation}
    \label{eq:spectral-gap}
    \ap{ h, Lh }_{-1,\frac{2}{\rho}}
    \leq
    - \norm{h}_{-1,\frac{2}{\rho}}^2,
  \end{equation}
  where $\ap{ \cdot,  \cdot }_{-1,\frac{2}{\rho}}$ is the scalar product
  associated to the norm $\norm{\cdot}_{-1,2/\rho}$,
  \begin{equation*}
    \ap{ h, Lh }_{-1,\frac{2}{\rho}}
    :=
    \int_0^\infty H(y) \mathcal{L}h (y) e^{-\frac{2}{\rho} y} \,dy.
  \end{equation*}
\end{prp}

\begin{proof}
  Set $\mu := 2/\rho$ as a shorthand. We explicitly calculate this
  expression by using eq. \eqref{eq:primitive-L}:
  \begin{multline}
    \label{eq:gap-1}
    \ap{ h, Lh }_{-1,\mu}
    =
    \int_0^\infty H(y) \int_y^\infty Lh (x)\,dx\, e^{\mu y} \,dy
    \\
    =
    - \int H^2 e^{\mu y}
    - \int H h\,y e^{\mu y}
    + \int H (g_\rho * H) e^{\mu y}.
  \end{multline}
  We rewrite these terms using integration by parts:
  \begin{equation}
    \label{eq:gap-3}
    \int H h \,y \, e^{\mu y}
    =
    - \int H h \,y \, e^{\mu y}
    + \int H^2 e^{\mu y}
    + \mu \int H^2 \,y \,e^{\mu y},
  \end{equation}
  and hence
  \begin{equation}
    \label{eq:gap-4}
    \int H h \,y \, e^{\mu y}
    =
    \frac{1}{2} \int H^2 \, e^{\mu y}
    + \frac{\mu}{2} \int H^2 \,y \,e^{\mu y}.
  \end{equation}
  As for the third term in eq. \eqref{eq:gap-1},
  \begin{equation}
    \label{eq:gap-2}
    \int_0^\infty H (g_\rho * H) e^{\mu y}
    =
    \frac{4}{\rho}
    \int_0^\infty H(y) \int_0^y H(x) e^{\mu x} \,dx \,dy
    =
    \frac{4}{\rho}
    \int H \mathcal{H} e^{\mu y}
    ,
  \end{equation}
  where $\mathcal{H}$ is the primitive of $H$:
  \begin{equation*}
    \mathcal{H}(y) := \int_y^\infty H(x) \,dx.
  \end{equation*}
  Putting eqs. \eqref{eq:gap-4} and \eqref{eq:gap-2} into \eqref{eq:gap-1}
  one obtains
  \begin{equation*}
    \ap{ h, Lh }_{-1,\mu}
    =
    - \frac{3}{2} \int H^2 e^{\frac{2}{\rho} y}
    - \frac{1}{\rho} \int H^2 \,y \,e^{\frac{2}{\rho} y}
    + \frac{4}{\rho} \int \mathcal{H} H e^{\frac{2}{\rho} y}.
  \end{equation*}
  Using the linearized Aizenman-Bak inequality (Lemma
  \ref{lem:Linear-Aizenman-Bak}) with $H$ instead of $h$ and $\mu :=
  2/\rho$ gives
  \begin{multline*}
    \ap{ h, Lh }_{-1,\mu}
    \leq
    - \frac{3}{2} \int H^2 e^{\frac{2}{\rho} y}
    - \frac{1}{\rho} \int H^2 \,y \,e^{\frac{2}{\rho} y}
    \\
    + \frac{1}{\rho} \int H^2 \,y \,e^{\frac{2}{\rho} y} +
    \frac{1}{2} \int H^2 e^{\frac{2}{\rho} y}
    =
    - \int H^2 e^{\frac{2}{\rho} y} 
    =
    - \norm{h}_{-1,\mu}^2
    ,
  \end{multline*}
  where we have taken into account that $\int y\, h = 0$ when omitting
  one of the terms in the inequality.
\end{proof}

\section{Spectral gaps of the linearized operator}
\label{sec:spectral gaps}

\subsection{Functional spaces and definition of $L$}
\label{sec:functional spaces}

For $\mu > 0$ and integer $k \geq -1$, we consider the following norms
for a function $h:(0,\infty) \to \RR$, defined when the expression
below make sense:
\begin{gather}
  \label{eq:def:Hk-mu-norm}
  \norm{h}_{k,\mu}^2
  := \int_0^\infty y^{2(k+1)} (D^k h(y))^2 e^{\mu y} \,dy.
\end{gather}
Here, $D^k h$ denotes the $k$-th derivative of $h$; for $D^1 h$ we
often use the more common $h'$, and when $k = -1$ we set $D^{-1} h :=
-H$ (cf. eq.~\eqref{eq:dfn-H-intro}). We also consider the following
spaces:
\begin{align}
  \label{eq:X_-1,mu}
  &X_{-1,\mu} := \left\{
    H' \mid H \in L^2(e^{\mu y})
    \quad \text{and} \quad
    \int H = 0
  \right\}
  \\
  \label{eq:X_k,mu}
  &X_{k,\mu} := \left\{
    h \mid 
    y^{k+1} D^k h \in L^2(e^{\mu y})
    \quad \text{and} \quad
    \int y\,h = 0
  \right\},
\end{align}
where it is understood that the members of these spaces are at least
distributions on $(0, \infty)$, and the derivatives are meant also in
the sense of distributions. In each $X_{k,\mu}$ for integer $k \geq
-1$ and $\mu > 0$, the above defined $\norm{\cdot}_{k,\mu}$ gives a
Hilbert norm, and their associated scalar products are denoted by
$\ap{\cdot,\cdot}_{k,\mu}$. We will prove later that the linearized
coagulation operator has a spectral gap in these spaces, and
considering only functions $h$ with $\int y \,h = 0$ is tied to the
fact that the coagulation equation conserves the mass, so its
linearization cannot have a spectral gap in a space without a
restriction of this kind. Observe that the restriction $\int H = 0$ in
$X_{-1,\mu}$ is the same as $\int y\, H' = - \int y\, h = 0$ when $H$
is regular enough, so it is the same as in the other spaces.

For given $\mu > 0$, the spaces $X_{k,\mu}$ form a scale where each
space is dense in the next one, as we show in the next two lemmas.

\begin{lem}
  \label{lem:Hardy-L2}
  For any $\mu \geq 0$, any integer $n \geq 0$ and $h \in L^2(y^{2n+2}
  e^{\mu y})$ it holds that
  \begin{equation}
    \label{eq:Hardy-L2}
    \int_0^\infty H^2 y^{2n} e^{\mu y}\,dy
    \leq
    4 \int_0^\infty h^2 y^{2(n+1)} e^{\mu y}\,dy.
  \end{equation}
\end{lem}

\begin{proof}
  Notice that $h$ is integrable on $(\epsilon, \infty)$ for any
  $\epsilon > 0$, so $H$ is well-defined and $H(y) \to 0$ as $y \to
  +\infty$. We use Cauchy-Schwarz's and Hardy's inequality as follows:
  \begin{multline*}
    \int_0^\infty H^2 y^{2n} e^{\mu y} \,dy
    =
    \int_0^\infty \int_y^\infty H(y) h(x) y^{2n} e^{\mu y} \,dx \,dy
    \\
    =
    \int_0^\infty h(x) \int_0^x H(y) y^{2n} e^{\mu y} \,dy \,dx
    \\
    \leq
    \int_0^\infty \abs{h(x)} x^{n+1} \, e^{\frac{\mu}{2}x}
    \frac{1}{x} \int_0^x \abs{H(y)} y^n e^{\frac{\mu}{2} y} \,dy \,dx
    \\
    \leq
    \left(
      \int_0^\infty h^2 x^{2n+2} e^{\mu x} \,dx
    \right)^{\frac{1}{2}}
    \left(
      \int_0^\infty
      \left(
        \frac{1}{x}
        \int_0^x H y^n e^{\frac{\mu}{2} y} \,dy
      \right)^2
      \,dx
    \right)^{\frac{1}{2}}
    \\
    \leq
    2
    \left(
      \int_0^\infty h^2 x^{2n+2} e^{\mu x} \,dx
    \right)^{\frac{1}{2}}
    \left(
      \int_0^\infty
      H^2 y^{2n} e^{\mu y} \,dy
    \right)^{\frac{1}{2}}.
  \end{multline*}
  This proves the inequality in the lemma.
\end{proof}

\begin{lem}
  \label{lem:X_k_mu scale}
  For $\mu > 0$ and integer $k \geq 0$, $X_{k,\mu}$ is contained and
  dense in $X_{k-1,\mu}$.
\end{lem}

\begin{proof}
  Lemma \ref{lem:Hardy-L2} proves the inclusion. The claim that they
  are dense inclusions is a consequence of the fact that the set $\{ f
  \in \mathcal{C}^\infty_0([0,\infty)) \mid \int y f = 0 \}$ is dense
  in all of them.
\end{proof}

Let us also give other useful inequalities:

\begin{lem}
  \label{lem:weighted-Poincare}
  For $\mu > 0$ and any $h \in L^2(e^{\mu y})$,
  \begin{equation}
    \label{eq:weighted-Poincare}
    \int H^2 e^{\mu y}
    \leq
    \frac{4}{\mu^2}
    \int h^2 e^{\mu y}.
  \end{equation}
\end{lem}

\begin{proof}
  The inequality $(2h - \mu H)^2 \geq 0$ and an integration by parts
  gives
  \begin{multline*}
    4 \int h^2 e^{\mu y}
    \geq
    4 \mu \int h H e^{\mu y}
    - \mu^2 \int H^2 e^{\mu y}
    \\
    =
    2 \mu \left( \int h \right)^2
    +
    \mu^2 \int H^2 e^{\mu y}
    \geq
    \mu^2 \int H^2 e^{\mu y}.
  \end{multline*}
\end{proof}


A consequence of Lemma \eqref{eq:weighted-Poincare} applied to the
function $y^{k+1} D^k H$ in the place of $H$, and Lemma
\ref{lem:Hardy-L2}, is the following inequality:

\begin{lem}
  \label{lem:weighted-Poincare-X_k,mu}
  For integer $k \geq 0$ and $\mu > 0$, there is a number $C =
  C(k,\mu) > 0$ such that
  \begin{equation*}
    \norm{H}_{k,\mu} \leq C \norm{h}_{k,\mu}
  \end{equation*}
  for all $h \in X_{k,\mu}$.
\end{lem}


Another useful property of these spaces is that exchanging derivatives
and powers of $y$ gives other equivalent norms. More precisely:

\begin{lem}
  \label{lem:equivalence of norms}
  Take $\mu > 0$, and nonnegative integers $k,a,b,n,m$ such that $a+b
  = n+m = k+1$. Assume that $h \in \mathcal{C}^{k-1}$ is such that
  \begin{equation}
    \label{eq:condition-at-0}
    \lim_{y \to 0} y^{i+2} D^i h(y) = 0
    \qquad
    (i=0, \dots, k-1).
  \end{equation}
  Then there is a (constructive) constant $K = K(k,a,b,n,m) > 0$ such
  that
  \begin{equation*}
    \norm{ y^{a} D^k (y^b h)}_{L^2(e^{\mu y})}
    \leq
    K \norm{ y^{n} D^k (y^m h) }_{L^2(e^{\mu y})}
  \end{equation*}
  for any $h \in L^2(e^{\mu y})$, in the sense that the left hand side
  is finite whenever the right hand side is, and then the inequality
  holds.
\end{lem}

One can prove this by repeated application of weighted Hardy
inequalities (cf. \cite[Example 0.3]{citeulike:2926130}), and we omit
the details.

\begin{rem}
  From this it is easy to see that the $L^2(e^{\mu y})$-norm of any
  combination of products by positive powers of $y$ and derivatives of
  $h$ gives an equivalent norm as long as the total sum of all the
  powers and that of the orders of all derivatives is fixed. It is
  important to restrict these considerations to functions $h$
  satisfying \eqref{eq:condition-at-0} and having some integrability
  property at $+\infty$, (e.g., $h \in L^2(e^{\mu y})$ in our case),
  as otherwise the finiteness of the right hand term does not imply
  that of the left one.
\end{rem}

The above equivalence of norms can be used in our $X_{k,\mu}$ spaces
for $\mu,k > 0$, as the condition in the previous lemma holds:

\begin{lem}
  If $h \in X_{k,\mu}$ with $\mu > 0$ and integer $k > 0$, then
  \begin{equation}
    \label{eq:limits at 0}
    \lim_{y \to 0} y^{i+2} D^i h(y) = 0
    \qquad
    (i=0, \dots, k-1).
  \end{equation}
\end{lem}

\begin{proof}
  Take an integer $0 \leq i < k$. As $h$ is in $X_{i+1,\mu} \subset
  X_{i,\mu}$, one directly sees that the first derivative of $y^{i+2}
  D^i h(y)$ is integrable near $y = 0$. Hence, $y^{i+2} D^i h(y)$ has
  a limit at $y=0$. As $y^{i+1} D^i h(y)$ is integrable, it must be
  that this limit is $0$.
\end{proof}

One can define the operator $L$ in the spaces $X_{k,\mu}$, of course
agreeing with our previous one in Definition \ref{dfn:L-prev} whenever
both are applicable. For the rest of Section \ref{sec:spectral gaps}
we fix $\rho > 0$, and consider the linearization of the nonlinear
operator in eq. (\ref{eq:ss-coag-lambda=0}) around the self-similar
profile $g_\rho$:

\begin{dfn}
  Take $\mu > 0$. For $h \in X_{0,\mu}$ (in particular, for $h \in
  X_{k,\mu}$ with $k \geq 0$) we define $Lh$ by eq.
  \eqref{eq:simpler-L}, where the derivative of $h$ is taken in the
  sense of distributions.
\end{dfn}

Our aim in this section is to show that the operator $L$ is defined as
$L: X_{k+1,\mu} \to X_{k,\mu}$ for $k \geq -1$, and that it is a
closed operator in the spaces $X_{k,\mu}$ for $\mu > 0$ and any
integer $k \geq -1$. Before proving this we will need two lemmas in
which we study the operator $C(g, h)$ which appears in the definition
of $L$.

\begin{lem}
  \label{lem:Dk(yk g*h))}
  For $\mu > 0$, integer $k \geq 1$ and $g,h \in X_{k,\mu}$, it holds
  that
  \begin{equation*}
    D^k (y^{k+1} (g*h))
    =
    \sum_{i=0}^{k+1} \binom{k+1}{i}
    (D^{k+1-i} (y^{k+1-i} g)) * (D^{i-1}(y^{i}h)).
  \end{equation*}
  In fact, this holds when $y^k D^{k+1} g, y^{k+1} D^k h \in
  L^2(e^{\mu y})$; this is, when $g,h$ satisfy the conditions for
  being in $X_{k,\mu}$, but without imposing that $\int y\,h = \int
  y\,g = 0$. In particular, this applies to $g = g_\rho$ and $h \in
  X_{k,\mu}$.
\end{lem}

\begin{proof}
  One has
  \begin{equation*}
    y^{k+1} (g*h) (y) = \int_0^y y^{k+1} g(x) h(y-x) \,dx,
  \end{equation*}
  and the identity in the lemma can be obtained by writing the
  binomial expansion for $y^{k+1} = (x + (y-x))^{k+1}$ and
  differentiating the resulting convolutions. The border terms which
  appear when differentiating the integrals between $0$ and $y$ vanish
  due to \eqref{eq:limits at 0}.
\end{proof}

\begin{lem}
  \label{lem:linear C bounded}
  For $0 < \mu < 4/\rho$ and integer $k \geq -1$ there is a constant
  $K = K(\rho,k,\mu)$ such that
  \begin{equation*}
    \norm{C(h,g_\rho)}_{k,\mu} \leq K \norm{h}_{k,\mu}
    \quad \text{ for all } h \in X_{k,\mu}.
  \end{equation*}
\end{lem}

\begin{proof}
  Here, $C(h,g_\rho)$ is understood to be defined by
  eq. \eqref{eq:simpler-C}, which makes sense for $h \in X_{k,\mu}$
  (which implies $H \in L^2(e^{\mu y})$). Numbers $K_1, K_2,\dots$
  below are assumed to depend only on $\rho, k$ and $\mu$.

  Let us first prove the lemma for $k \geq 0$. From
  \begin{equation*}
    2 C(g_\rho, h)
    = -2h -\frac{4}{\rho} H + \frac{2}{\rho} H * g_\rho
  \end{equation*}
  and our previous lemma, we deduce that
  \begin{equation}
    \label{eq:lcb1}
    2 \norm{C(g_\rho,h)}_{k,\mu}
    \leq
    2 \norm{h}_{k,\mu}
    + \frac{4}{\rho} \norm{H}_{k,\mu}
    + \frac{2}{\rho} \norm{H * g_\rho}_{k,\mu}.
  \end{equation}
  The second term in eq. \eqref{eq:lcb1} is bounded by $K_1
  \norm{h}_{k,\mu}$ thanks to Lemma
  \ref{lem:weighted-Poincare-X_k,mu}, for some $K_1 > 0$. For the
  third term in \eqref{eq:lcb1} we have:
  \begin{multline*}
    \norm{H * g_\rho}_{k,\mu}
    \leq
    K_2
    \norm{ D^k \left( y^{k+1} (H * g_\rho) \right) }_{L^2(e^{\mu y})}
    \\
    \leq
    K_3
    \sum_{i=0}^{k+1} \binom{k+1}{i}
    \norm{h}_{i-1,\mu}
    \int \abs{D^{k-i}(y^{k-i}g_\rho)} y^{k-i} e^{\frac{\mu}{2}y}
    \\
    \leq
    K_4
    \norm{h}_{k,\mu}
  \end{multline*}
  for some constants $K_2, K_3, K_4 > 0$, where we have used Lemma
  \ref{lem:equivalence of norms} in the first inequality (observe that
  the conditions at $0$ are met), Lemma \ref{lem:Dk(yk g*h))} in the
  second, and Lemma \ref{lem:Hardy-L2} for the third one. As the
  integrals bounded in the last inequality are finite as long as $0
  \leq \mu < 4/\rho$, the lemma is proved for $k \geq 0$.

  Now, for $k = -1$, by density it is enough to prove the inequality
  when $h$ is $\mathcal{C}^\infty_0([0,\infty))$. We use the
  expression of the primitive of $C(g,h)$ from Lemma
  \ref{lem:primitive_C} to obtain, using Young's inequality, that
  \begin{multline}
    \label{eq:C(g_rho,h)_bounded_in_(-1,mu)}
    2 \norm{C(g_\rho,h)}_{-1,\mu}
    \leq
    \norm{g_\rho * H}_{L^2(e^{\mu y})} + \norm{h}_{-1,\mu} \int g_\rho
    \\
    \leq
    \norm{H}_{L^2(e^{\mu y})} \int \abs{g_\rho} e^{\frac{\mu}{2} y}
    + \norm{h}_{-1,\mu} \int g_\rho
    \leq
    2 \norm{h}_{-1,\mu} \int \abs{g_\rho} e^{\frac{\mu}{2} y}.
  \end{multline}
\end{proof}

We finally have the following:

\begin{prp}
  \label{prp:L closed unbounded}
  For any integer $k \geq -1$ and $0 < \mu < 4/\rho$, the operator $L$
  is defined between $X_{k+1,\mu}$ and $X_{k,\mu}$:
  \begin{equation*}
    L: X_{k+1,\mu} \to X_{k,\mu}.
  \end{equation*}
  Seen as an unbounded operator on $X_{k,\mu}$ for integer $k \geq -1$
  and real $\mu > 0$, $L$ is a closed operator with dense domain. In
  addition, for every $\mu > 0$ and integer $k \geq -1$ there is a
  constant $K = K(\rho,k,\mu)$ such that
  \begin{equation}
    \label{eq:L generates a semigroup}
    \ap{Lh,h}_{k,\mu}
    \leq
    K \norm{h}_{k,\mu}^2
    \qquad (h \in X_{k+1,\mu}),
  \end{equation}
  and consequently $L$ generates an evolution semigroup in each of the
  spaces $X_{k,\mu}$.
\end{prp}

\begin{proof}
  To see that $L$ is defined between $X_{k+1,\mu}$ and $X_{k,\mu}$ for
  any $k \geq -1$, take $h \in X_{k+1,\mu}$ and notice that every term
  in \eqref{eq:dfn:L} is in $X_{k,\mu}$, as can be seen from Lemma
  \ref{lem:linear C bounded} and the fact that $h \in X_{k+1,\mu}$
  implies that $y\,h' \in X_{k,\mu}$ (see Lemma \ref{lem:equivalence
    of norms} and the remark that follows).

  For $k \geq -1$, let us prove that $L: X_{k+1,\mu} \to X_{k,\mu}$ is
  a closed operator in the space $X_{k,\mu}$. Take a sequence $h_n$ in
  $X_{k+1,\mu}$ which converges to some $h \in X_{k,\mu}$ in the norm
  of $X_{k,\mu}$, and such that $Lh_n \to \tilde{h}$ in $X_{k,\mu}$
  for some $\tilde{h} \in X_{k,\mu}$. We need to show that $\tilde{h}
  \in X_{k+1,\mu}$ and $L h = \tilde{h}$.  For such a sequence we have
  \begin{equation*}
    L h_n = 2 h_n + y h_n' + 2 C(g_\rho, h_n)
  \end{equation*}
  As the sequence $\{h_n\}$ converges to $h$, it is clear from Lemma
  \ref{lem:linear C bounded} that the first and last terms in the
  expression of $Lh_n$ converge to $2 h$ and $2 C(g_\rho, h)$,
  respectively, in $X_{k,\mu}$. As $Lh_n$ converges to $\tilde{h}$, it
  follows that the second term, $y h_n'$, converges to something in
  $X_{k,\mu}$. From the equivalence of norms in Lemma
  \ref{lem:equivalence of norms}, this implies that in fact $h_n$
  converges to something in the norm of $X_{k+1,\mu}$. As this is a
  stronger norm than $\norm{\cdot}_{k,\mu}$, we see that the limit
  must be $h$, which implies that $h \in X_{k+1,\mu}$ and $L h_n \to L
  h$ in $X_{k,\mu}$. Hence, $\tilde{h} = Lh$.

  Finally, let us prove the inequality \eqref{eq:L generates a
    semigroup} for $k \geq 0$ (for $k=-1$ a similar argument proves
  it). From expression \eqref{eq:dfn:L} we have
  \begin{multline}
    \label{eq:sg1}
    \ap{Lh,h}_{k,\mu}
    =
    2 \norm{h}_{k,\mu}^2
    + \ap{ y h', h}_{k,\mu}
    + 2 \ap{C(g_\rho, h), h}_{k,\mu}
    \\
    \leq
    2 \norm{h}_{k,\mu}^2
    + \ap{ y h', h}_{k,\mu}
    + 2 \norm{C(g_\rho, h)}_{k,\mu} \norm{h}_{k,\mu}
    .
  \end{multline}
  For the second term, as $D^k (y h') = k D^k h + y D^{k+1} h$,
  \begin{equation}
    \label{eq:sg2}
    \ap{ y h', h}_{k,\mu}
    =
    k \norm{h}_{k,\mu}^2
    +
    \int y^{2k+3} (D^{k+1}h) (D^k h) e^{\mu y}
    \leq
    k \norm{h}_{k,\mu}^2,
  \end{equation}
  as an integration by parts shows that the term we omitted is
  negative:
  \begin{multline}
    \label{eq:sg3}
    \int y^{2k+3} (D^{k+1}h) (D^k h) e^{\mu y}
    \\
    =
    - \frac{2k+3}{2}
    \int y^{2k+2} (D^{k}h)^2 e^{\mu y}
    - \frac{\mu}{2}
    \int y^{2k+3} (D^{k}h)^2 e^{\mu y}
    \\
    =
    - \frac{2k+3}{2} \norm{h}_{k,\mu}^2
    - \frac{\mu}{2}
    \int y^{2k+3} (D^{k}h)^2 e^{\mu y}
    .
  \end{multline}
  For the last term in \eqref{eq:sg1}, using Lemma
  \ref{lem:linear C bounded},
  \begin{equation}
    \label{eq:sg4}
    \norm{C(g_\rho, h)}_{k,\mu}
    \leq K \norm{h}_{k,\mu}
  \end{equation}
  Hence, we finally obtain
  \begin{equation*}
    \ap{Lh, h}_{k,\mu} \leq K_1 \norm{h}_{k,\mu}^2
  \end{equation*}
  for some $K_1$ depending on $\rho$ and $\mu$.
\end{proof}

\subsection{Spectral gap in $X_{-1,2/\rho}$}
\label{sec:gap-H-1}

A direct consequence of Proposition \ref{prp:spectral-gap} and a limit
argument is the exponential decay of the evolution semigroup defined
by $L$:

\begin{lem}
  \label{lem:spectral gap X_-1,2/rho}
  For $h \in X_{0,2/\rho}$,
  \begin{equation*}
    \ap{ h, Lh }_{-1,2/\rho} \leq - \norm{h}_{-1,2/\rho}^2.
  \end{equation*}
  As a consequence, for $h^0 \in X_{-1,2/\rho}$ we have
  \begin{equation}
    \label{eq:semigroup-decay-H-1}
    \norm{e^{t L} h^0}_{-1,2/\rho}
    \leq
    \norm{h^0}_{-1,2/\rho} e^{-t}
    \qquad (t \geq 0).
  \end{equation}
\end{lem}

\subsection{Extension of spectral gaps to lower exponential weights}
\label{sec:spectral gap extension}

The evolution semigroup generated by $L$ in the spaces $X_{k,\mu}$ has
the remarkable property of creating exponential moments of $(D^k h)^2$
in finite time. This will allow us to prove that, if $L$ has a
spectral gap in $X_{k,\nu}$ for some $k \geq -1$, then it also has a
spectral gap in $X_{k,\mu}$ for any $0 < \mu \leq \nu$.

\begin{lem}
  \label{lem:creation_exponential_moments}
  Take $0 < \mu < \nu < 4/\rho$ and an integer $k \geq -1$. If $h$ is
  a solution of the linear self-similar equation $\partial_t h = Lh$
  with $h(0) \in X_{k,\mu}$, then the norm $\norm{h(t)}_{k,\nu}$ is
  finite at time $t = t_0 := \log (\nu / \mu)$, and
  \begin{equation*}
    \norm{h(t_0)}_{k,\nu}
    \leq
    \left( \frac{\nu}{\mu} \right)^{K}
    \norm{h(0)}_{k,\mu}
  \end{equation*}
  for some positive constant $K = K(\rho, k, \mu, \nu)$.
\end{lem}

\begin{proof}
  Let us first prove the lemma for $k \geq 0$. From eqs.
  \eqref{eq:sg1}--\eqref{eq:sg4} in the proof of Proposition
  \ref{prp:L closed unbounded}, for any $0 < \gamma < 4/\rho$ we have
  \begin{equation}
    \label{eq:p-exp-moments}
    \ap{h,Lh}_{k,\gamma}
    \leq
    K \norm{h}_{k,\gamma}^2
    - \frac{\gamma}{2}
    \int y^{2k+3} (D^{k}h)^2 e^{\gamma y},
  \end{equation}
  where the constant $K = K(\rho,k,\gamma) > 0$ is the same as that of
  Proposition \ref{prp:L closed unbounded}. With this, take $\phi(t)
  := \mu e^{t}$ and carry out the following computation:
  \begin{multline}
    \label{eq:p-exp-moments-2}
    \frac{d}{dt} \norm{h}_{k,\phi(t)}^2
    =
    2 \ap{h, Lh}_{k,\phi(t)}
    + \phi'(t) \int y^{2k+3} (D^{k}h)^2 e^{\mu y}
    \\
    \leq
    2 K(\phi(t)) \norm{h}_{\phi(t)}^2
    + (\phi'(t) - \phi(t)) \int y^{2k+3} (D^{k}h)^2 e^{\mu y}
    \\
    =
    2 K(\phi(t)) \norm{h}_{\phi(t)}^2,
  \end{multline}
  where $K(\phi(t))$ is the constant obtained from
  \eqref{eq:p-exp-moments}, for which we write the dependence on
  $\phi(t)$ explicitly. Now, for $t \leq t_0$,
  \begin{equation}
    \frac{d}{dt}
    \norm{h}_{\phi(t)}^2
    \leq
    2 K(\phi(t)) \norm{h}_{\phi(t)}^2
    \leq
    K_1 \norm{h}_{\phi(t)}^2
  \end{equation}
  for some $K_1 > 0$, as the constant $K(\phi(t))$ is bounded for $0
  \leq t \leq t_0$, which can be checked from the proof of Proposition
  \ref{prp:L closed unbounded}. Hence, as $\phi(0) = \mu$,
  \begin{equation*}
    \norm{h}_{\phi(t_0)}^2
    \leq
    e^{K_1 t_0} \norm{h^0}_{\mu}^2
    =
    \left(\frac{\nu}{\mu}\right)^{K_1}  \norm{h^0}_{\mu}^2,
  \end{equation*}
  which proves the lemma for $k \geq 0$, given that $\phi(t_0) = \mu
  e^{t_0} = \nu$.

  For $k = -1$ the same argument can be carried out by using the
  inequality
  \begin{equation*}
    \ap{h,Lh}_{-1,\gamma}
    \leq
    K \norm{h}_{-1,\gamma}^2
    - \frac{\gamma}{2}
    \int y\,H^2 e^{\gamma y}
  \end{equation*}
  instead of \eqref{eq:p-exp-moments}, and the following instead of
  \eqref{eq:p-exp-moments-2}:
  \begin{equation*}
    \frac{d}{dt} \norm{h}_{-1, \phi(t)}^2
    =
    2 \ap{h, Lh}_{-1,\phi(t)}
    + \phi'(t) \int y\, e^{\phi(t) y} H^2.
  \end{equation*}
\end{proof}

\begin{lem}
  \label{lem:spectral gap extension}
  Take $0 < \nu < 4/\rho$ and an integer $k \geq -1$. Assume that the
  operator $L$ has a spectral gap in $X_{k,\nu}$ of size $\delta >
  0$; this is, there exists $C \geq 1$ such that
  \begin{equation*}
    \norm{e^{tL} h^0}_{k,\nu}
    \leq C \norm{h^0}_{k,\nu} e^{-\delta t}.
  \end{equation*}
  Then, $L$ has a spectral gap of the same size in $X_{k,\mu}$ for any
  $0 < \mu < \nu$; this is, there is $C' = C'(\mu, C) \geq 1$ such
  that
  \begin{equation*}
    \norm{e^{tL} h^0}_{k,\mu}
    \leq C' \norm{h^0}_{k,\mu} e^{-\delta t}.
  \end{equation*}
\end{lem}

\begin{proof}
  Take $h^0$ in the domain of $L$ as an operator on $X_{k,\mu}$ (i.e.,
  $X_{k+1,\mu}$ if $k \geq 0$, and $L^2(y\,e^{\mu y})$ if $k=-1$).
  Consider the solution $t \mapsto h(t)$ to eq.
  (\ref{eq:ss-coag-lambda=0}) with initial condition $h^0$. The idea
  is to estimate $\norm{h(t)}_{k,\mu}$ using eq.  \eqref{eq:L
    generates a semigroup} until a time $t_0$ for which
  $\norm{h(t_0)}_{-1,\nu}$ is finite, and after that time use the fact
  that the latter norm is exponentially decreasing thanks to the
  spectral gap we are assuming in the norm of $X_{k,\nu}$.

  Choose $t_0 := \log \frac{\nu}{\mu}$. Then, for $t \leq t_0$, the
  estimate in eq. \eqref{eq:L generates a semigroup} gives
  \begin{equation*}
    \frac{d}{dt} \norm{h(t)}_{k,\mu}^2
    =
    2 \ap{h(t), Lh(t)}_{k,\mu}
    \leq
    K_{1}
    \norm{h(t)}_{k,\mu}^2.
  \end{equation*}
  So, for $0 \leq t \leq t_0$ we have
  \begin{equation}
    \label{eq:estimate_before_t0}
    \norm{h(t)}_{k,\mu}^2
    \leq
    \norm{h^0}_{k,\mu}^2 e^{K_1 t}
    \leq
    \norm{h^0}_{k,\mu}^2 e^{K_1 t_0}
    =
    \norm{h^0}_{k,\mu}^2 \left( \frac{\nu}{\mu} \right)^{K_1}
  \end{equation}
  At $t=t_0$ we can use Lemma \ref{lem:creation_exponential_moments}
  with $\nu := 2/\rho$ to get that
  \begin{equation*}
    \norm{h(t_0)}_{k,\nu}^2
    \leq
    \norm{h^0}_{k,\mu}^2 \left( \frac{\nu}{\mu} \right)^K,
  \end{equation*}
  where $K$ is the one from the lemma. Now, as we are assuming a
  spectral gap of a certain size $\delta$ in $X_{k,\nu}$, we have for
  $t \geq t_0$
  \begin{multline*}
    \norm{h(t)}_{k,\mu}^2 \leq \norm{h(t)}_{k,\nu}^2 \leq
    \norm{h(t_0)}_{k,\nu}^2 e^{-2\delta (t-t_0)}
    \\
    \leq \norm{h^0}_{k,\mu}^2 \left( \frac{\nu}{\mu} \right)^K
    e^{-2\delta (t-t_0)}
    =
    \norm{h^0}_{k,\mu}^2 \left( \frac{\nu}{\mu}
    \right)^{K+2\delta} e^{-2\delta t}
    .
  \end{multline*}
  Together with eq. \eqref{eq:estimate_before_t0}, this shows that
  \begin{equation}
    \label{eq:spectral_gap_mu}
    \norm{h(t)}_{k,\mu}
    \leq
    C \norm{h^0}_{k,\mu} e^{-\delta t}
    \quad \text{ for all } t \geq 0,
  \end{equation}
  with
  \begin{equation*}
    C^2 := \max \left\{
      \left( \frac{\nu}{\mu} \right)^{K+2\delta},
      \left( \frac{\nu}{\mu} \right)^{K_1}
    \right\}.
  \end{equation*}
  The result is extended to all $h^0 \in X_{k,\mu}$ by density.
\end{proof}

Then, from the spectral gap we showed in Lemma \ref{lem:spectral gap
  X_-1,2/rho} we obtain the following:

\begin{cor}
  \label{cor:spectral gap (-1,mu) and (0,mu)}
  Take $0 < \mu \leq 2/\rho$, and consider the semigroup $e^{tL}$
  generated by $L$ in the space $X_{-1, \mu}$. Then there is a number
  $K = K(\mu) \geq 1$ such that
  \begin{equation*}
    \norm{e^{tL} h^0}_{-1,\mu}
    \leq K \norm{h^0}_{-1,\mu} e^{- t}
    \quad
    \text{ for any } h^0 \in X_{-1,\mu}.
  \end{equation*}
\end{cor}

\subsection{Extension of the spectral gap to $X_{k,\mu}$ with $k \geq 0$}
\label{sec:spectral-gap-derivative}

Let us shortly discuss the strategy of the proofs in this section. One
could try the following argument in order to obtain a spectral gap for
a norm involving a certain derivative $D^k h$ of order $k \geq 0$.  We
could hope that, when calculating
\begin{equation*}
  \frac{d}{dt} \int_0^\infty (D^k h (y))^2 e^{\frac{2}{\rho} y} \,dy,
\end{equation*}
one obtains a negative multiple of this same norm, plus some terms
which involve $L^2$ norms of $h$ with some weight, and which can
already be controlled by the semigroup decay results in previous
sections; this idea can be arrived at by trying it for $k=0$, in which
case it works perfectly. But for $k > 0$, the term which involves the
same norm is actually a \emph{positive} multiple of it, so our
tentative argument does not directly give a spectral gap in this case.
It is for this reason that we consider instead norms with a power
weight $y^m$ for a certain $m$: one is forced to raise the power $m$
if one wants to obtain a spectral gap involving $L^2$ norms of higher
derivatives of $h$. Then, another problem appears: if one calculates,
say, the time derivative of $\int y^m (h')^2 e^{\mu y}$, then some
terms which involve $L^2$ norms of $h$ will appear \emph{with the same
  weight} $y^m$ as we were using in the first place: this is, a term
involving $\int h^2 y^m e^{\mu y}$ will appear, for which we have no
previous spectral gap results. Hence, to close the estimates we use a
spectral gap in a norm with an exponential weight of order slightly
higher than the one we started with.

\begin{lem}
  \label{lem:semigroup_decay_Hk}
  Take $0 < \mu < 2/\rho$ and an integer $k \geq 0$, and consider the
  operator $L$ in the space $X_{k,\mu}$, with domain
  $X_{k+1,\mu}$. Then, $L$ has a spectral gap in $X_{k,\mu}$; more
  precisely, there is some constant $K = K(k, \mu, \rho) > 0$ such
  that
  \begin{equation}
    \label{eq:semigroup-decay-mu}
    \norm{e^{tL} h^0}_{k,\mu} \leq K \norm{h^0}_{k,\mu} e^{-t}
  \end{equation}
  for any $h^0 \in X_{k,\mu}$.
\end{lem}

\begin{proof}
  Take $k \geq 0$, assume the result true up to $k-1$ (the case $k =
  -1$ was proved in Corollary \ref{cor:spectral gap (-1,mu) and
    (0,mu)}). Denote by $h$ the solution of the linear equation with
  initial condition $h^0$ (i.e., $h(t,\cdot) = e^{tL} h^0$), and in
  order to simplify the notation write $\norm{\cdot}$ instead of
  $\norm{\cdot}_{k,\mu}$ and $\ap{\cdot, \cdot}$ instead of
  $\ap{\cdot,\cdot}_{k,\mu}$. The numbers $K, K_1, K_2,\dots$ which
  appear in this proof are understood to be positive and depend only
  on $k,\mu,\rho$ and $\delta$.

  Pick $\mu < \nu < 2/\rho$. Then, after the time $t_0$ given in Lemma
  \ref{lem:creation_exponential_moments}, the norm
  $\norm{h(t)}_{k,\nu}$ is finite and we have, from the expression for
  $L$ in eq. \eqref{eq:simpler-L}:
  \begin{equation}
    \label{eq:spk1}
    \frac{1}{2}
    \frac{d}{dt} \norm{h}^2
    =
    \ap{h, Lh}
    =
    \ap{h, y h'}
    - \frac{4}{\rho} \ap{h, H}
    + \frac{2}{\rho} \ap{h, H*g_\rho}.
  \end{equation}
  For the first term in \eqref{eq:spk1},
  eqs. \eqref{eq:sg2}--\eqref{eq:sg3} show that
  \begin{equation}
    \label{eq:spk2}
    \ap{h, y h'}
    \leq
    - \frac{3}{2} \norm{h}^2.
  \end{equation}

  For the middle term in eq. \eqref{eq:spk1},
  \begin{multline}
    \label{eq:spk-middle}
    \ap{h, H}
    = \int (D^k h) (D^k H) y^{2k+2} e^{\mu y} \,dy
    \\
    = - \int (D^k h) (D^{k-1} h) y^{2k+2} e^{\mu y} \,dy
    \\
    =
    (k+1) \int (D^{k-1} h)^2 y^{2k+1} e^{\mu y} \,dy
    \\
    + \frac{\mu}{2} \int (D^{k-1} h)^2 y^{2k+2} e^{\mu y} \,dy
    \leq
    K_1 \norm{h}_{k-1,\nu}^2,
  \end{multline}
  where in the case $k = 0$ we denote $D^{-1} h := -H$.  Observe that
  here the exponential weight has been changed from $\mu$ to $\nu$, as
  the power weights in the norms which appear are not the right ones
  for the $\norm{\cdot}_{k-1,\mu}$ norm, but higher. We overcome this
  difficulty by using a norm with a slightly higher exponential
  weight.
  
  To bound the middle term in \eqref{eq:spk1}, we will need the
  following expression for $D^k (h * g_\rho)$, which can easily be
  proved by induction:
  \begin{equation}
    \label{eq:spk-Dh*g}
    D^k(H*g_\rho)
    =
    H * (D^k g_\rho)
    + \sum_{i=0}^{k-1} (D^{i-1} h) (D^{k-1-i} g_\rho)(0).
  \end{equation}
  From this,
  \begin{multline}
    \label{eq:spk7}
    \norm{H * g_\rho}
    =
    \norm{ y^{k+1} D^k(H*g_\rho)}_{L^2(e^{\mu y})}
    \\
    \leq
    \norm{y^{k+1} (H * (D^k g_\rho))}_{L^2(e^{\mu y})}
    + \sum_{i=0}^{k-1} \abs{(D^{k-1-i} g_\rho)(0)}
    \norm{y^k D^{i-1} h}_{L^2(e^{\mu y})}
    \\
    \leq
    \norm{y^{k+1} (H * (D^k g_\rho))}_{L^2(e^{\mu y})}
    + K_2 \norm{h}_{k-1,\nu}
    ,
  \end{multline}
  where the last bound is possible because all the norms inside the
  sum are bounded by a constant times $\norm{h}_{k,\nu}$, using a
  higher exponential weight as before. For the remaining term in
  \eqref{eq:spk7},
  \begin{equation*}
    y^{k+1} (H * (D^k g_\rho))
    =
    \sum_{i=0}^{k+1}
    \binom{k+1}{i}
    (y^i H) * (y^{k+1-i} D^k g_\rho))
  \end{equation*}
  and hence
  \begin{multline*}
    \norm{y^{k+1} (H * (D^k g_\rho))}_{L^2(e^{\mu y})}
    \leq
    \\
    \sum_{i=0}^{k+1}
    \binom{k+1}{i}
    \norm{(y^i H) * (y^{k+1-i} D^k g_\rho))}_{L^2(e^{\mu y})}
    \\
    \leq
    \sum_{i=0}^{k+1}
    \binom{k+1}{i}
    \norm{y^i H}_{L^2(e^{\mu y})} \int y^{k+1-i}
    \abs{D^k g_\rho} e^{\frac{\mu}{2}y}
    \\
    \leq
    K_3
    \norm{h}_{-1,\nu}    
    \leq
    K_4
    \norm{h}_{k-1,\nu}    
    .
  \end{multline*}
  Notice that the integrals involving $g_\rho$ are numbers that depend
  only on $k,\mu$ and $\rho$. From \eqref{eq:spk7},
  \begin{equation*}
    \norm{H*g_\rho}
    \leq
    K_5 \norm{h}_{k-1,\nu},
  \end{equation*}
  so
  \begin{equation}
    \label{eq:spk8}
    \abs{\ap{h, H*g_\rho}}
    \leq
    \norm{h} \norm{H*g_\rho}
    \leq
    K_5 \norm{h} \norm{h}_{k-1,\nu}.
  \end{equation}
  Finally, all three terms from \eqref{eq:spk1} are bounded in
  \eqref{eq:spk2}, \eqref{eq:spk-middle} and \eqref{eq:spk8}, and
  together give
  \begin{multline*}
    \frac{1}{2}
    \frac{d}{dt} \norm{h}^2
    \leq
    - \frac{3}{2} \norm{h}^2
    + \frac{4}{\rho} K_1 \norm{h}_{k-1,\nu}^2
    + \frac{2}{\rho} K_5 \norm{h} \norm{h}_{k-1,\nu}
    \\
    \leq
    - \frac{5}{4} \norm{h}^2
    + K_6 \norm{h}_{k-1,\nu}^2
    ,
  \end{multline*}
  where we used the elementary inequality $\norm{h} \norm{h}_{k-1,\nu}
  \leq \epsilon^2 \norm{h}^2 + (1/\epsilon) \norm{h}_{k-1,\nu}^2$ for
  an arbitrary $\epsilon > 0$.
  Thanks to the spectral gap in $\norm{\cdot}_{k-1,\nu}$ and Lemma
  \ref{lem:creation_exponential_moments}, we know that for any $\delta
  < 1$ there is some constant $K > 0$ such that $\norm{h(t)}_{k-1,\nu}
  \leq K \norm{h(t_0)}_{k-1,\nu} e^{-\delta (t-t_0)}$ for $t \geq
  t_0$. Hence, for $t \geq t_0$,
  \begin{multline*}
    \frac{1}{2}\frac{d}{dt} \norm{h}^2
    \leq
    - \frac{5}{4} \norm{h}^2
    + K_7 \norm{h^0}_{k-1,\nu}^2 e^{-2 \delta(t-t_0)}
    \\
    \leq
    - \frac{5}{4} \norm{h}^2
    + K_8 \norm{h^0}^2 e^{-2 \delta t}
  \end{multline*}
  A Gronwall lemma, together with inequality (\ref{eq:L generates a
    semigroup}) for $t \leq t_0$, finishes the proof.
\end{proof}

\section{Behavior of moments}
\label{sec:behavior-moments}

For later use we will need to have some information on the time
evolution of exponential moments of solutions to the coagulation
equation \eqref{eq:ss-coag-lambda=0}, which is also interesting by
itself, as this equation is one of the few particular cases where one
can find an explicit expression for the evolution of moments. For a
function $g: (0,\infty) \to \RR$ we denote
\begin{gather}
  \label{eq:dfn-potential-moments}
  M_k[g] := \int_0^\infty y^k \abs{g(y)} \,dy
  \quad \text{ for } k \in \RR,
  \\
  \label{eq:dfn-exp-moments}
  E_\mu[g] := \int_0^\infty e^{\mu y} \abs{g(y)} \,dy
  \quad \text{ for } \mu \in \RR.
\end{gather}
In general it is slightly simpler to solve explicitly the evolution of
moments in eq. \eqref{eq:coag-eq} than in eq.
\eqref{eq:ss-coag-lambda=0}; but of course, one of them being a
rescaling of the other, there is a simple relationship between the
moments of their solutions: if $f$ and $g$ are related by the change
of variables in eqs.
\eqref{eq:change-forward}--\eqref{eq:change-backward}, then the
relationship between moments of $f$ and moments of $g$ is
\begin{gather}
  \label{eq:change-moments-fw}
  M_k[g(t)]
  =
  e^{t(1-k)} \,M_k[f (e^{t} - 1)]
  \\
  \label{eq:change-moments-bw}
  M_k[f(t)]
  =
  (t+1)^{k-1}
  \,M_k[g\big( \log(1+t) \big)].
\end{gather}
That between exponential moments of $f$ and $g$ is
\begin{gather}
  \label{eq:change-exp-moments-fw}
  E_\mu[g(t)]
  =
  e^t \,E_{\mu e^{-t}} [f
  \big( e^{t} - 1 \big) ],
  \\
  \label{eq:change-exp-moments-bw}
  E_\mu[f(t)]
  =
  \frac{1}{t+1}
  \,E_{\mu (t+1)} [g\big( \log(t+1) \big)]. 
\end{gather}

The evolution of moments of order $0$, $1$ and $2$ for a nonnegative
solution $f$ to eq. \eqref{eq:coag-eq} (or the corresponding
solution $g$ to eq. \eqref{eq:ss-coag-lambda=0}) is easily obtained
and well-known: the first moment is a constant, called its $mass$. For
the moment of order 0,
\begin{equation*}
  \frac{d}{dt} M_0[f] = -\frac{1}{2} M_0[f]^2,
\end{equation*}
so
\begin{equation}
  \label{eq:explicit-M0}
  M_0[f] = \frac{2}{t + 2/M_0[f^0]}.
\end{equation}
Using the relation in eq. \eqref{eq:change-moments-fw} (with $\lambda
= 0$) we have
\begin{equation}
  \label{eq:explicit-M0-ss}
  M_0[g]
  =
  \frac{2}{1 - e^{-t} + 2\,e^{-t}/M_0[g^0]}.
\end{equation}
The second moment of $f$ in eq. \eqref{eq:coag-eq} with $a \equiv 1$
satisfies the equation
\begin{equation*}
  \frac{d}{dt} M_2[f]
  =
  \frac{1}{2} M_1[f]^2 = \frac{1}{2} \rho^2,
\end{equation*}
where $\rho$ is the mass of the solution. Hence,
\begin{equation}
  \label{eq:explicit-M2}
  M_2[f] = M_2[f^0] + \frac{1}{2}\rho^2 t,
\end{equation}
and the second moment for the self-similar equation is
\begin{equation}
  \label{eq:explicit-M2-ss}
  M_2[g] = e^{-t} M_2[g^0] + \frac{1}{2} \rho^2 (1 - e^{-t}).
\end{equation}
For exponential moments one can also find an explicit expression:
again from eq. \eqref{eq:coag-eq}, and using eq.
\eqref{eq:explicit-M0},
\begin{equation*}
  \begin{split}
    \frac{d}{dt} E_\mu[f]
    &= \frac{1}{2} E_\mu[f]^2 - M_0[f] E_\mu[f]
    \\
    &= \frac{1}{2} E_\mu[f]^2 - \frac{2}{t + K} E_\mu[f],
  \end{split}
\end{equation*}
with $K := 2/M_0[f^0]$. This has an explicit solution:
\begin{equation}
  \label{eq:explicit-Emu}
  E_\mu[f(t)]
  =
  \frac{2}{t + \frac{2}{M_0^0}}
  +
  \frac{2}{\frac{2}{E_\mu^0 - M_0^0} - t},
\end{equation}
where $M_0^0$ and $E_\mu^0$ denote $M_0[f(0)]$ and $E_\mu[f(0)]$,
resp.

Using eq. \eqref{eq:change-exp-moments-fw} we obtain the evolution of
exponential moments for the self-similar equation with $\lambda = 0$:
\begin{equation}
  \label{eq:explicit-Emu-ss}
  E_\mu[g(\log(s))]
  =
  \frac{2s}{s-1 + \frac{2}{M_0^0}}
  +
  \frac{2s}{\frac{2}{E_{\mu/s}^0 - M_0^0} - s+1}.
\end{equation}
This directly implies the following lemma:
\begin{lem}
  \label{lem:exp-moment-finite}
  If $g$ is a solution of eq. \eqref{eq:ss-coag-lambda=0} with initial
  data $g^0$ and $\mu > 0$, then $E_\mu[g]$ is finite for all $t > 0$
  if and only if the initial data satisfies
  \begin{equation}
    \label{eq:condition-exp-moment-finite}
    {E_{\theta}^0 - M_0^0} < \frac{2}{\frac{\mu}{\theta}-1}
    \quad
    \text{ for all } 0 < \theta < \mu.
  \end{equation}
  $E_\mu[g]$ is uniformly bounded for all $t \geq 0$ if and only if
  the initial data satisfies the following for some $\nu > \mu$:
  \begin{equation}
    \label{eq:condition-exp-moment-uniformly-bounded}
    {E_{\theta}^0 - M_0^0} \leq \frac{2}{\frac{\nu}{\theta}-1}
    \quad
    \text{ for all } 0 < \theta < \mu.
  \end{equation}
  In particular, if $E_\nu[g]$ is finite for all $t > 0$, then for
  each $\mu < \nu$ the moment $E_\mu[g]$ is uniformly bounded for all
  $t \geq 0$.
\end{lem}

Next we prove that if \emph{some} positive exponential moment is
initially finite, then every exponential moment less than
$\frac{2}{\rho}$ becomes bounded after some time. The precise result
is the following:

\begin{lem}
  \label{lem:Emu-finite-after-T}
  Let $g$ be a solution of eq. \eqref{eq:ss-coag-lambda=0} with
  initial data $g^0$. Assume that there exists $\mu > 0$ such that
  $E_\mu[g^0] < \infty$. Then, for each $0 < \nu < \frac{2}{\rho}$
  there is a time $T_\nu \geq 0$ (which depends on the initial
  condition $g^0$) such that
  \begin{equation}
    \label{eq:Emu-finite-after-T}
    E_\nu[g(t)] < \infty
    \quad \text{ for all } t > T_\nu.
  \end{equation}
  As a consequence, for each $0 < \nu < \frac{2}{\rho}$ there is a
  time $T_\nu^* \geq 0$ and a constant $K_\nu > 0$ such that
  \begin{equation}
    \label{eq:Emu-bounded-after-T}
    E_\nu[g(t)] \leq K_\nu
    \quad \text{ for all } t \geq T_\nu^*.
  \end{equation}
\end{lem}

\begin{proof}
  The uniform bound eq. \eqref{eq:Emu-bounded-after-T} is a
  consequence of eq. \eqref{eq:Emu-finite-after-T} in view of Lemma
  \ref{lem:exp-moment-finite}, so we will only prove
  eq. \eqref{eq:Emu-finite-after-T}. Take $0 < \nu <
  \frac{2}{\rho}$. In view of eq. \eqref{eq:explicit-Emu-ss}, it is
  enough to show that, for some $\mu > \epsilon > 0$,
  \begin{equation}
    \label{eq:proof-Emu-1}
    E_\theta[g^0] - M_0[g^0]
    <
    \frac{2}{\frac{\nu}{\theta} - 1}
    \quad
    \text{ for all }
    \theta \leq \epsilon,
  \end{equation}
  and then eq. \eqref{eq:Emu-finite-after-T} holds with $T_\nu := \log
  \frac{\nu}{\epsilon}$. As eq. \eqref{eq:proof-Emu-1} is a local
  inequality, it is enough to study both terms near $\theta = 0$: they
  are $0$ at $\theta = 0$, and their derivatives are:
  \begin{gather*}
    \frac{d}{d\theta}
    \int_0^\infty g(y) e^{\theta y} \,dy
    \Bigg \vert_{\theta = 0}
    =
    \int_0^\infty y g(y) \,dy = \rho
    \\
    \frac{d}{d\theta}
    \frac{2}{\frac{\nu}{\theta} - 1}
    \Bigg \vert_{\theta = 0}
    =
    \frac{2\nu}{(\nu - \theta)^2}
    \Bigg \vert_{\theta = 0}
    =
    \frac{2}{\nu}.
  \end{gather*}
  Hence, the inequality
  \begin{equation*}
    \nu < \frac{2}{\rho}
  \end{equation*}
  shows that eq. \eqref{eq:proof-Emu-1} holds for some $\epsilon > 0$,
  and the result is proved.
\end{proof}

\section{Local exponential convergence}
\label{sec:local}

In this section we will prove Theorem \ref{thm:local exponential
  convergence-intro}, which is a \emph{local} exponential convergence
result of the following kind: we show that if a solution $g$ to eq.
\eqref{eq:ss-coag-lambda=0} is initially \emph{close enough} to the
stationary solution, then it converges exponentially fast to it with a
given rate in the norm $\norm{\cdot}_{k,\mu}$ for $0 < \mu < 2/\rho$
and integer $k \geq -1$, as long as this norm is finite at time $t=0$.

Here, the concept of \emph{close enough} is measured in the sense that
the initial relative entropy to the equilibrium is small. The reason
for this is that this closeness must be enough to guarantee that the
nonlinear term in the evolution equation is small when compared to the
linear part, which we know is well-behaved. However, if $g$ is a
solution, the smallness of the nonlinear part essentially involves
norms of the kind $\int \abs{g-g_\rho}$, which cannot be controlled by
our norm $\norm{g-g_\rho}_{-1,2/\rho}$. Consequently, we need to
measure closeness in some way that is propagated in time: what is
needed in the proof of the theorem is that, for $\epsilon > 0$ small
enough,
\begin{equation}
  \label{eq:L1 norm condition}
  \int_0^\infty \abs{g(t,y)-g_\rho(y)} \,dy
  \leq \epsilon
  \quad \text{ for all } t \geq 0.
\end{equation}
We do not know of any simple condition on the initial data that
guarantees this to hold except requiring the relative entropy to the
equilibrium $g_\rho$ to be small. Then, as the $L^1$ norm of
$g-g_\rho$ is controlled by the relative entropy, we know that
(\ref{eq:L1 norm condition}) also holds. When we consider convergence
in norms $\norm{\cdot}_{k,\mu}$ with $k \geq 0$ we could control $\int
\abs{g-g_\rho}$ by the norm $\norm{g-g_\rho}_{k,\mu}$ and obtain a
result where closeness to $g_\rho$ is measured in terms of
$\norm{g-g_\rho}_{k,\mu}$, but this is less convenient for the global
result in Section \ref{sec:global}. Because of this, we always state
our results with an entropy condition.

Before proving our local convergence theorems we will need two
previous lemmas. In the first one we show that the function $F[g \vert
g_\rho]$ controls the distance of $g$ to the equilibrium $g_\rho$ in
the $L^1$ norm, and the second one will be used to estimate the
nonlinear term $C(g-g_\rho, g-g_\rho)$ which will appear when
calculating the time evolution of $\norm{g-g_\rho}_{k,\mu}$.

\begin{lem}
  \label{lem:particle_difference_estimate}
  There is some constant $K \geq 0$ such that
  \begin{equation*}
    \int_0^\infty \abs{ g(y) - g_\rho(y) } \,dy
    \leq
    K \max \left\{
      \sqrt{F[g \vert g_\rho]}, F[g \vert g_\rho]
    \right\}.
  \end{equation*}
\end{lem}

\begin{proof}
  We can write
  \begin{equation*}
    F[g \vert g_\rho] =
    \int_0^\infty g_\rho(y)
    \Psi \left(\frac{g(y)}{g_\rho(y)} - 1\right) \,dy,
  \end{equation*}
  where $\Psi(y) := (y+1) \log(y+1) - y \geq 0$. As it happens that
  $\Psi$ is a convex function for which $\Psi(y) \geq \Psi(\abs{y})$,
  from Jensen's inequality we have
  \begin{multline*}
    F[g \vert g_\rho]
    \geq
    \int_0^\infty g_\rho(y)
    \Psi \left( \abs{
        \frac{g(y)}{g_\rho(y)} - 1
      } \right) \,dy
    \\
    \geq 2 \,
    \Psi \left(
      \frac{1}{2} \int_0^\infty \abs{g(y) - g_\rho(y)} \,dy
    \right),
  \end{multline*}
  where the 2 appears from $\int_0^\infty g_\rho = 2$. Hence,
  \begin{equation*}
    \int \abs{g - g_\rho}
    \leq
    2 \Psi^{-1}\left(
      \frac{1}{2} F[g \vert g_\rho]
    \right)
    \leq
    K \max \left\{
      \sqrt{F[g \vert g_\rho]}, F[g \vert g_\rho]
    \right\}
  \end{equation*}
  for some constant $K \geq 0$, which proves the lemma.
\end{proof}

\begin{lem}
  \label{lem:C(h,h) bound}
  For $\rho > 0$ and $0 < \mu < 4/\rho$ there is a constant $K =
  K(\mu) > 0$ such that
  \begin{equation}
    \label{eq:C(h,h) bound - (-1,mu)}
    \norm{C(h,h)}_{-1,\mu}
    \leq K \norm{h}_{-1,\mu} \int \abs{h} e^{\frac{\mu}{2}y} \,dy
    \qquad (h \in X_{-1,\mu}).
  \end{equation}
  Also, for integer $k \geq 0$ and $\mu < \nu < 4/\rho$ , there is a
  constant $K = K(k,\mu,\nu) > 0$ such that
  \begin{equation}
    \label{eq:C(h,h) bound - (k,mu)}
    \norm{C(h,h)}_{k,\mu}
    \leq
    K \norm{h}_{k,\mu} \int \abs{h} e^{\frac{\nu}{2}y} \,dy
    + K \norm{h}_{k-1,\nu}^2
  \end{equation}
  for all $h \in X_{k,\nu}$.
\end{lem}

\begin{proof}
  Eq. (\ref{eq:C(h,h) bound - (-1,mu)}) is a consequence of the
  calculation in eq. (\ref{eq:C(g_rho,h)_bounded_in_(-1,mu)}) with $h$
  instead of $g_\rho$. Now, for $k \geq 1$ and $h \in X_{k,\mu}$, from
  the expression $C(h,h) = \frac{1}{2} h * h - h \int h$ we have
  \begin{equation}
    \label{eq:Chh1}
    \norm{C(h,h)}_{k,\mu}
    \leq
    \frac{1}{2} \norm{h * h}_{k,\mu}
    + \norm{h}_{k,\mu} \abs{\int h},
  \end{equation}
  and from Lemmas \ref{lem:equivalence of norms} and \ref{lem:Dk(yk
    g*h))},
  \begin{multline}
    \label{eq:Chh2}
    \norm{h*h}_{k,\mu}
    \leq
    K_1 \norm{ D^k (y^{k+1} (h*h)) }_{L^2(e^{\mu y})}
    \\
    \leq
    K_1
    \sum_{i=0}^{k+1} \binom{k+1}{i}
    \norm{(D^{i-1} (y^i h)) * (D^{k+1-i}(y^{k+1-i}h))}_{L^2(e^{\mu y})}
    \\
    \leq
    K_2 \norm{h}_{k,\mu} \int \abs{h} e^{\frac{\mu}{2}y}
    + K_3 \norm{h}_{k,\mu} \int \abs{h} y\,e^{\frac{\mu}{2}y}
    \\
    + K_4
    \sum_{i=2}^{k} \binom{k+1}{i}
    \norm{h}_{i,\mu} \int \abs{D^{k+1-i}(y^{k+1-i}h)} e^{\frac{\mu}{2}y}
    \\
    \leq
    K_5 \norm{h}_{k,\mu} \left(
      \int \abs{h} e^{\frac{\mu}{2}y}
      + \int \abs{h} y\,e^{\frac{\mu}{2}y}
    \right)
    + K_6 \norm{h}_{k-1,\nu}^2.
  \end{multline}
  Here, we have bounded the terms in the sum using Young's inequality;
  the terms for $i=0$, $i=1$ and $i=k+1$ have been bounded putting the
  $L^2$ norm on the part with $k$ derivatives (and an integration by
  parts for $i=0$). The rest of the terms have been also bounded using
  Young's inequality, and we have applied Cauchy-Schwarz for the last
  inequality, together with the fact that $\norm{h}_{k-1,\nu}$
  controls all norms $\norm{\cdot}_{i,\nu}$ with $0 \leq i < k$.
  Notice that $K_1$ through $K_5$ depend only on $k$ and $\mu$, and
  $K_6$ only on $k,\mu$ and $\nu$. Finally, from (\ref{eq:Chh1}),
  \begin{multline*}
    \norm{C(h,h)}_{k,\mu}
    \leq
    \norm{h}_{k,\mu} \abs{\int h}
    \\
    +
    K_5 \norm{h}_{k,\mu}
    \left(
      \int \abs{h} e^{\frac{\mu}{2}y}
      + \int \abs{h} y\,e^{\frac{\mu}{2}y}
    \right)
    + K_6 \norm{h}_{k-1,\nu}^2
    \\
    \leq
    K_7 \norm{h}_{k,\mu} \int \abs{h} e^{\frac{\nu}{2}y}
    + K_6 \norm{h}_{k-1,\nu}^2.
  \end{multline*}
  This proves our result for $k \geq 1$. For the case $k = 0$, the
  same calculation proves the inequality, with the difference that in
  this case the term multiplying $K_4$ in (\ref{eq:Chh2}) does not
  appear and the argument is much simpler.
\end{proof}
We can finally prove Theorem \ref{thm:local exponential
  convergence-intro}:

\begin{proof}[Proof of Thorem \ref{thm:local exponential
    convergence-intro}]
  \textbf{Step 1: Proof for $k=-1$}. Denote $\norm{\cdot}_{-1,\mu}$ by
    $\norm{\cdot}$. We assume that $\norm{g^0 - g_\rho} < \infty$, and
    want to show the exponential convergence in the same norm. We
    prove an \emph{a priori} estimate below, which can be made
    rigorous by considering approximate regular solutions instead of
    solutions to the full equation. The numbers $K, K_1, K_2,\dots$
    which appear in the proof are understood to depend only on $\rho,
    E, \delta, k, \mu$ and $\nu$.

  Consider a solution $g$ of the self-similar equation eq.
  (\ref{eq:ss-coag-lambda=0}) in the conditions of the theorem, and
  set $h := g - g_\rho$, $h^0 := g^0 - g_\rho$. Then $h$ is a solution
  of the equation $\partial_t h = L(h) + C(h,h)$, and by Duhamel's
  formula we have
  \begin{equation}
    \label{eq:Duhamel-0}
    h(t)
    = e^{tL} h^0
    + \int_0^t e^{(t-s)L} \left[ C(h(s), h(s)) \right] \,ds
    \quad
    (t \geq 0).
  \end{equation}
  Using our bound on the exponential decay of the semigroup $e^{tL}$
  from Corollary \ref{cor:spectral gap (-1,mu) and (0,mu)} we obtain
  \begin{equation*}
    \norm{h(t)}
    \leq
    K_1 \norm{h^0} e^{- t}
    +
    K_1 \int_0^t \norm{ C(h(s), h(s)) } e^{- (t-s)} \,ds,
  \end{equation*}
  for some $K_1 > 0$ which depends on $\rho, E$ and $\mu$. Now, use the bound in Lemma \ref{lem:C(h,h) bound} to get
  \begin{equation}
    \label{eq:Duhamel-1}
    \norm{h(t)}
    \leq
    K_1 \norm{h^0} e^{- t}
    +
    K_2 \int_0^t
    E_{\frac{\mu}{2}}[h(s)]
    \norm{ h(s) }  e^{-(t-s)} \,ds
    ,
  \end{equation}
  where $E_{\frac{\mu}{2}}[h(s)]$ is the exponential moment of order
  $\mu/2$ of $h$ (defined in
  eq. \eqref{eq:dfn-exp-moments}). Considering $E$ and $\nu$ from the
  hypotheses of the theorem, we have $E_{\nu/2}[g(t)] \leq E$ for all $t
  \geq 0$, so there is some constant $C_1 > 0$ such that
  \begin{equation*}
    E_{\nu/2}[h(t)] \leq C_1
    \quad \text{ for all } t > 0,
  \end{equation*}
  (as $h = g-g_\rho$, and $g_\rho(y) e^{\frac{\nu}{2} y}$ is integrable). By
  Hölder's inequality,
  \begin{equation}
    \label{eq:nonlinear-L2-1}
    E_{\frac{\mu}{2}}[h(s)]
    \leq
    M_0[h(s)]^{\frac{\nu - \mu}{\nu}}
    E_{\nu/2}[h(s)]^{\frac{\mu}{\nu}}
    \leq
    \epsilon_1
    \,
    C_2
    =:
    \epsilon_2
    \quad
    \text{ for } s \geq 0,
  \end{equation}
  where
  \begin{equation*}
    C_2 := C_1^{\frac{\mu}{\nu}},
    \qquad
    \epsilon_1^{\frac{\nu}{\nu - \mu}}
    := K \max \left\{ \sqrt{\epsilon}, \epsilon \right\},
  \end{equation*}
  and $K$ is the one from Lemma
  \ref{lem:particle_difference_estimate}. We have used here the bound
  on the relative entropy in the hypotheses, and the fact that it is
  conserved in time, as the relative entropy is nonincreasing.

  Using this bound on eq. \eqref{eq:Duhamel-1} we have
  \begin{equation}
    \label{eq:Duhamel-1.1}
    \norm{h}
    \leq
    K_1 \norm{h^0} e^{- t}
    +
    \epsilon_3
    \int_0^t
    \norm{ h(s) }  e^{- (t-s)} \,ds,
  \end{equation}
  where $\epsilon_3 := \epsilon_2 K_2$. Equivalently,
  \begin{equation*}
    \norm{h} e^{t}
    \leq
    K_1 \norm{h^0}
    +
    \epsilon_3
    \int_0^t
    \norm{ h(s) }  e^s \,ds,
  \end{equation*}
  and by Gronwall's lemma,
  \begin{equation*}
    \norm{h} e^{t}
    \leq
    K_1 \norm{h^0} e^{\epsilon_3 t},
  \end{equation*}
  or
  \begin{equation*}
    \norm{h}
    \leq
    K_1 \norm{h^0} e^{-(1 - \epsilon_3) t}.
  \end{equation*}
  Choosing $\epsilon$ so that $1 - \epsilon_3 \geq \delta$
  proves the theorem.

  \textbf{Step 2: Proof for $k \geq 0$.} For this step, denote by
  $\norm{\cdot}$ the norm $\norm{\cdot}_{k,\mu}$. The difference in
  this case is that the inequality (\ref{eq:C(h,h) bound - (k,mu)})
  has an additional term, which can be dealt with by using our
  previous steps. Take $k \geq 0$, and assume that we have proved the
  result on the convergence in $\norm{\cdot}_{k-1,\gamma}$ for $0 <
  \gamma < 2/\rho$. Repeating the same reasoning as before we arrive
  at the following instead of (\ref{eq:Duhamel-1.1}):
  \begin{multline}
    \label{eq:Duhamel-2}
    \norm{h(t)}
    \leq
    K_1 \norm{h^0} e^{- t}
    +
    \epsilon_3
    \int_0^t
    \norm{ h(s) }  e^{- (t-s)} \,ds
    \\
    + 
    K_2 \int_0^t
    \norm{h(s)}_{k-1,\nu}^2
    e^{- (t-s)} \,ds
    ,
  \end{multline}
  where the inequality \eqref{eq:C(h,h) bound - (k,mu)} was applied
  with some $\nu'$ such that $\mu < \nu' < \nu$ in the place of $\nu$.
  Now, as we know that $\norm{h(s)}_{k-1,\nu}$ is exponentially
  decaying in time, for any $\delta < \delta' < 1$ we have
  \begin{equation*}
    \norm{h(s)}_{k-1,\nu}
    \leq
    K_3 \norm{h^0}_{k-1,\nu} e^{-\delta s}.
  \end{equation*}
  Hence, we get
  \begin{multline*}
    \norm{h(t)}
    \leq
    K_1 \norm{h^0} e^{- t}
    \\
    +
    \epsilon_3
    \int_0^t
    \norm{ h(s) }  e^{-(t-s)} \,ds
    + 
    K_4 \norm{h^0}_{k-1,\nu}
    e^{-\delta' t}
    \\
    \leq
    K_5 \norm{h^0}_{k,\nu} e^{- {\delta'} t}
    +
    \epsilon_3
    \int_0^t
    \norm{ h(s) }  e^{-{\delta'} (t-s)} \,ds,
  \end{multline*}
  using that $\norm{h}_{k,\nu}$ controls both $\norm{h}_{k-1,\nu}$ and
  $\norm{h}_{k,\mu}$. Now this inequality has the same form as
  (\ref{eq:Duhamel-1.1}), and the same reasoning followed in the first
  step finishes the proof.
\end{proof}

\section{Global convergence}
\label{sec:global}

To prove global convergence we consider two stages in the evolution of
the equation: when the solution is close enough to equilibrium, we can
use Theorem \ref{thm:local exponential convergence-intro} to get an
exponential rate of convergence; when it is far from equilibrium we
use the following lemma, which says that the solution converges to
equilibrium exponentially fast in the $L^2$ norm:

\begin{lem}
  \label{lem:L2 explicit convergence}
  Take a nonnegative initial condition $g_0 \in L^1_2$, with
  derivative $g_0' \in L^1$ and mass $\rho > 0$, and consider the
  solution $g$ to eq.  \eqref{eq:ss-coag-lambda=0}.  There are
  explicit constants $C, T > 0$ which depend only on $g_0$ such that
  \begin{equation}
    \label{eq:exp convergence}
    \norm{g(t) - g_\rho}_2 \leq C e^{-\frac{1}{2}t}
    \qquad (t > T).
  \end{equation}
\end{lem}

The lemma is independent from the rest of the paper, as it can be
proved by using the explicit solution of the Fourier transform of
equation \eqref{eq:ss-coag-lambda=0}. The technique to prove it is
very similar to the one used in \cite{MR2139564} to obtain the uniform
convergence to equilibrium; the main difference is that here we are
interested in \emph{obtaining a rate}, while in \cite{MR2139564}
emphasis was placed in giving very general conditions on the initial
data under which convergence holds, without any given rate.

Lemma \ref{lem:L2 explicit convergence} does not give the optimal rate
of convergence, but we only need it to show that any solution, after
some initial time which may be explicitly given in terms of the
initial datum, is in the conditions of our local result in Theorem
\ref{thm:local exponential convergence-intro}, and hence prove our global
result in Theorem \ref{thm:global-convergence-intro}. However, one
could be more careful in the proof below and, imposing slightly
stronger decay conditions on the initial datum $g_0$, obtain the
optimal rate $e^{-t}$.

We will first prove Theorem \ref{thm:global-convergence-intro}
assuming Lemma \ref{lem:L2 explicit convergence}, leaving the proof of
the lemma itself for the end.

\begin{proof}[Proof of Theorem \ref{thm:global-convergence-intro}]
  Let us show that, for any $\mu < \nu$, we can apply Theorem
  \ref{thm:local exponential convergence-intro} for the exponential decay of
  the norm $\norm{\cdot}_{k,\mu}$ after an certain initial time $t_0$.

  Lemma \ref{lem:exp-moment-finite} and an argument as in the proof of
  Lemma \ref{lem:Emu-finite-after-T} show that for some $0 < \mu' <
  \nu$, the exponential moment $E_{\mu'/2}[g]$ is uniformly bounded
  for all times $t \geq 0$. So, take $0 < \mu < \mu'$ with $\mu <
  2/\rho$. Then, the first condition in Theorem \ref{thm:local
    exponential convergence-intro} holds. Together with this, Lemma
  \ref{lem:L2 explicit convergence}, shows that the relative entropy
  $F[g|g_\rho]$ tends to zero at an explicit rate, and hence condition
  2 in Theorem \ref{thm:local exponential convergence-intro} is also
  satisfied after some initial time $t_0$.

  Finally, the same calculation as in Theorem \ref{thm:local
    exponential convergence-intro} shows that, for $0 < \mu < \mu'$
  with $\mu < 2/\rho$, the norm $\norm{g-g_\rho}_{k,\mu}$ is bounded
  in bounded time intervals, and in particular it is finite at $t =
  t_0$.  This is enough to apply the Theorem \ref{thm:local
    exponential convergence-intro}, so we have
  \begin{equation*}
    \norm{g(t)-g_\rho}_{k,\mu}
    \leq
    K_1 e^{-\delta (t-t_0)}
    \qquad
    (t \geq t_0),
  \end{equation*}
  for any $0 < \mu < \nu$ with $\mu < 2/\rho$. As
  $\norm{g(t)-g_\rho}_{k,\mu}$ is bounded for $t \in [0,t_0]$, this
  proves the result.
\end{proof}

\begin{proof}[Proof of Lemma \ref{lem:L2 explicit convergence}]
  By the scaling and time-translation symmetries of eq.
  \eqref{eq:coag-eq}, it is enough to prove the result when $\int y\,g
  = \int g = 2$ (see, e.g., \cite{MR2139564}). We consider the Fourier
  transform of $g$,
  \begin{equation}
    \label{eq:Fourier-g}
    \phi_t(\mu) := \int_0^\infty e^{-i \mu y} g(t,y) \,dy.
  \end{equation}
  A very similar calculation to that in section
  \ref{sec:behavior-moments} shows that $\phi$ has the following
  explicit expression, identical to \eqref{eq:explicit-Emu-ss} (notice
  that we are assuming $\int g_0 \equiv M_0^0 = 2$):
  \begin{equation}
    \label{eq:Fourier-explicit}
    \phi_t(\mu)
    =
    2
    +
    \frac{2\tau}{\frac{2}{\phi_0(\frac{\mu}{\tau}) - 2} - \tau+1}
    =
    \frac{2 \phi_0(\frac{\mu}{\tau})}
    {2 + (\tau - 1)(2 - \phi_0(\frac{\mu}{\tau}))}
    ,
  \end{equation}
  for $\mu \in \RR$ and $t > 0$, denoting $\tau := e^t$. By
  Plancherel's theorem, proving our result is equivalent to showing
  that
  \begin{equation}
    \label{eq:Fourier-exp convergence}
    \norm{\phi_t - \phi_\infty}_2 \leq C e^{-\frac{1}{2}t}
    \qquad (t > 0),
  \end{equation}
  where $\phi_\infty$ is the Fourier transform of the equilibrium with
  mass $2$, $g_2(y) := 2 e^{-y}$,
  \begin{equation}
    \label{eq:Fourier-equilibrium}
    \phi_\infty(\mu) := \frac{2}{1 + i\mu}
    \qquad (\mu \in \RR).
  \end{equation}

\paragraph{Step 1: Approximations of $\phi_0(x)$ at $x=0$}

As $\phi(0) = 2$ and
\begin{gather}
  \label{eq:phi_0'}
  \frac{d}{dx} \phi_0(x)
  = - i \int_0^\infty e^{-i x y} y \,g_0(y) \,dy,
  \\
  \label{eq:bound phi_0'}
  \abs{\frac{d}{dx} \phi_0(x)}
  \leq
  \int_0^\infty y \,g_0(y) \,dy = 2,
\end{gather}
we readily have
\begin{equation}
  \label{eq:bound phi_0-2}
  \abs{\phi_0(x) - 2}
  \leq 2 \abs{x}
  \qquad (x \in \RR).
\end{equation}
Similarly, $\phi'(0) = -2i$ and
\begin{gather}
  \label{eq:phi_0''}
  \frac{d^2}{dx^2} \phi_0(x)
  = - \int_0^\infty e^{-i x y} y^2 \,g_0(y) \,dy,
  \\
  \label{eq:bound phi_0''}
  \abs{\frac{d^2}{dx^2} \phi_0(x)}
  \leq
  \int_0^\infty y^2 \,g_0(y) \,dy
  =: M_2,
\end{gather}
and then, using Taylor's series at $x = 0$,
\begin{equation}
  \label{eq:bound difference quotient}
  \abs{\frac{\phi_0(x) - 2}{x} + 2i}
  \leq
  \frac{1}{2} M_2 \abs{x}
  \qquad
  (x \in \RR).
\end{equation}

\paragraph{Step 2: Lower bound of the denominator in
  \eqref{eq:Fourier-explicit}.}

For $\mu \in \RR$ and $\tau > 1$ we have, calling $x := \mu/\tau$,
\begin{equation}
  \label{eq:den1}
  \abs{2 + (\tau - 1)(2 - \phi_0(x))}
  \geq
  \abs{2 + 2i\mu}
  - \abs{2i\mu - (\tau -1)(2 - \phi_0(x))},
\end{equation}
and rewriting this second term gives
\begin{equation*}
  2i\mu - (\tau -1)(2 - \phi_0(x))
  =
  \mu \left(
    2i + \frac{\phi_0(x) - 2}{\mu/\tau}
  \right)
  +
  2 - \phi_0(x),
\end{equation*}
and using eqs. \eqref{eq:bound phi_0-2} and \eqref{eq:bound difference
  quotient},
\begin{equation}
  \label{eq:den2}
  \abs{ 2i\mu - (\tau -1)(2 - \phi_0(x)) }
  \leq
  \frac{1}{2} M_2 \abs{\mu} \abs{x}
  +
  2 \abs{x}
  \qquad (x \equiv \frac{\mu}{\tau} \in \RR),
\end{equation}
and so from \eqref{eq:den1},
\begin{equation*}
  \abs{2 + (\tau - 1)(2 - \phi_0(x))}
  \geq
  1 + \abs{\mu}
  -\frac{1}{2} M_2 \abs{\mu} \abs{x}
  -2 \abs{x}
  .
\end{equation*}
Finally, for $\abs{x} \leq \epsilon_1 := \min\left\{ 1/4, 1/M_2
\right\}$,
\begin{equation}
  \label{eq:den- small x estimate}
  \abs{2 + (\tau - 1)(2 - \phi_0(x))}
  \geq
  \frac{1}{2} (1 + \abs{\mu})
  \qquad (\abs{x} \equiv \abs{\frac{\mu}{\tau}} \leq \epsilon_1 )
  .
\end{equation}
When $\abs{x} \geq \epsilon_1$, we estimate the denominator the other
way round. Call $$\epsilon_2 := \inf_{\abs{x} \geq \epsilon_1} \abs{2 -
  \phi_0(x)} > 0.$$ The above defined $\epsilon_2$ is strictly
positive because $\phi_0$ is continuous, $\abs{\phi_0(x)} < 2$ for
$\abs{x} > 0$ and
\begin{equation}
  \label{eq:phi_0 decay}
  \abs{\phi_0(x)}
  \leq 2 C_1 \frac{1}{\abs{x}}
  \qquad (\abs{x} > 0),
\end{equation}
with
\begin{equation*}
  C_1 := \int_0^\infty \abs{g_0'(y)} \,dy.
\end{equation*}
Hence, for $\abs{x} \geq \epsilon_1$ the denominator can also be
estimated as
\begin{equation}
  \label{eq:den3}
  \abs{2 + (\tau - 1)(2 - \phi_0(x))}
  \geq
  (\tau -1)\abs{(2 - \phi_0(x))} - 2
  \geq
  (\tau - 1) \epsilon_2 - 2,
\end{equation}
and then for $\tau -1 \geq C_2 := 3/\epsilon_2$,
\begin{equation}
  \label{eq:den- large x estimate}
  \abs{2 + (\tau - 1)(2 - \phi_0(x))}
  \geq
  1
  \qquad  (\abs{x} \equiv \abs{\frac{\mu}{\tau}} \geq \epsilon_1,
  \quad \tau \geq C_2 + 1)
  .
\end{equation}
Putting together \eqref{eq:den- small x estimate} and \eqref{eq:den-
  large x estimate} we obtain
\begin{equation}
  \label{eq:den- full estimate}
  \abs{2 + (\tau - 1)(2 - \phi_0(\mu/\tau))}
  \geq
  \frac{1}{2}
  \qquad (\mu \in \RR, \quad \tau \geq C_2 + 1).
\end{equation}

\paragraph{Step 3: Estimate for $\abs{\mu}$ large.}

Take $K > 0$. We have, for $\mu \in \RR$ and $t \geq 0$,
\begin{equation}
  \label{eq:xl1}
  \abs{\phi_t(\mu) - \phi_\infty(\mu)}
  \leq
  \abs{\phi_t(\mu)} + \abs{\phi_\infty(\mu)},
\end{equation}
and
\begin{equation}
  \label{eq:xl2}
  \int_{\abs{\mu} > K} \abs{\phi_\infty(\mu)}^2 \,d\mu
  \leq
  \int_{\abs{\mu} > K} \frac{1}{\abs{1 + i\mu}^2}
  \leq
  \int_{\abs{\mu} > K} \frac{1}{\mu^2}
  =
  \frac{2}{K}.
\end{equation}
On the other hand, notice that \eqref{eq:den3} implies that for $\tau
\geq C_2 + 1$,
\begin{equation}
  \label{eq:xl2.5}
  \abs{2 + (\tau - 1)(2 - \phi_0(x))}
  \geq
  \epsilon_3 \tau
\end{equation}
for some $\epsilon_3$ depending only on $C_2$ and $\epsilon_2$. Then,
using the explicit expression \eqref{eq:Fourier-explicit}, eq.
\eqref{eq:xl2.5} and eq. \eqref{eq:phi_0 decay},
\begin{equation}
  \label{eq:xl3}
  \abs{\phi_t(\mu)} \leq
  \frac{2}{\epsilon_3} \frac{ \abs{\phi_0(\mu/\tau)} }
  {\tau}
  \leq
  4 \frac{C_1}{\epsilon_3} \frac{1}{\abs{\mu}}
  \qquad (\tau \geq C_2 + 1),
\end{equation}
so, for $\tau \geq C_2 + 1$,
\begin{equation}
  \label{eq:xl4}
  \int_{\abs{\mu} > K} \abs{\phi_t(\mu)}^2 \,d\mu
  \leq
  16 \frac{C_1^2}{\epsilon_3^2}
  \int_{\abs{\mu} > K} \frac{1}{\abs{\mu}^2} \,d\mu
  =
  \leq 32 \frac{C_1^2}{\epsilon_3^2} \frac{1}{K}.
\end{equation}
Equations \eqref{eq:xl2} and \eqref{eq:xl4} together now give
\begin{equation}
  \label{eq:xl5}
  \int_{\abs{\mu} > K} \abs{\phi_t - \phi_\infty}^2
  \,d\mu
  \leq
  \frac{1}{K} \left(
    2 + 32 \frac{C_1^2}{\epsilon_3^2}
  \right)
  \qquad
  (\tau \geq C_2 + 1, \quad K > 0).
\end{equation}

\paragraph{Step 4: Estimate for $\abs{\mu}$ small.}

Denoting $x \equiv \mu/\tau$, the difference $\phi_t(\mu) -
\phi_\infty(\mu)$ can be written as
\begin{equation}
 \label{eq:phi - phi_infty}
 \phi_t(\mu) - \phi_\infty(\mu)
 =
 \frac{
   2i\mu(\phi_0(x) - 2)
   + 2\mu \left(
     2i + \frac{\phi_0(x) - 2}{x}
   \right)
 }
 {(1+i\mu)(2 + (1-\tau)(\phi_0(x)-2))}
 .
\end{equation}
For the numerator we use \eqref{eq:bound phi_0-2} and \eqref{eq:bound
 difference quotient} to get
\begin{multline}
 \label{eq:xs1}
 \abs{2i\mu(\phi_0(x) - 2)
   + 2\mu \left(
     2i + \frac{\phi_0(x) - 2}{x}
   \right)
 }
 \\
 \leq
 (4 + M_2) \abs{\mu} \abs{x}
 = (4 + M_2) \abs{\mu}^2 \frac{1}{\tau}.
\end{multline}
For the denominator, we use eq. \eqref{eq:den- small x estimate}:
\begin{equation*}
 \abs{2 + (1-\tau)(\phi_0(x)-2)}
 \geq
 \frac{1}{2} (1 + \abs{\mu})
 \qquad (\abs{\frac{\mu}{\tau}} \leq \epsilon_1 ).
\end{equation*}
Then, for $\abs{\frac{\mu}{\tau}} \leq \epsilon_1$,
\begin{equation}
 \label{eq:xs2}
 \abs{\phi_t(\mu) - \phi_\infty(\mu)}
 \leq
 C_3 \frac{1}{\tau} \frac{\abs{\mu}^2}{(1 + \abs{\mu})^2}
 \leq
 C_3 \frac{1}{\tau},
\end{equation}
with $C_3 := 4 (4 + M_2)$.
Hence for any $K := \epsilon_1 \tau$ we have
\begin{multline}
 \label{eq:xs3.5}
 \int_{\abs{\mu} < K}
 \abs{\phi_t(\mu) - \phi_\infty(\mu)}^2 \,d\mu
 \leq
 C_3^2 \frac{1}{\tau^2} \int_{\abs{\mu} < K} \,d\mu
 \\
 =
 2 C_3^2 \frac{K}{\tau^2}
 = 2 C_3^2 \epsilon_1^2 \frac{1}{K}
 =: C_4 \frac{1}{K}.
\end{multline}

\paragraph{Step 5: Final estimate.}

From \eqref{eq:xl5} and \eqref{eq:xs3.5}, taking $K := \epsilon_1
\tau$ we obtain for $\tau \geq C_2 + 1$:
\begin{equation}
 \label{eq:fin1}
 \int_{-\infty}^{+\infty}
 \abs{\phi_t(\mu) - \phi_\infty(\mu)}^2 \,d\mu
 \leq
 \frac{1}{K} \left(
   C_4
   + 2 + 32 \frac{C_1^2}{\epsilon_3^2}
 \right)
 =: C_5  e^{-t}.
\end{equation}
This shows \eqref{eq:Fourier-exp convergence}, and finishes the proof.

\end{proof}

\bibliographystyle{abbrv}
\bibliography{ozarfreo-latin1}

\end{document}